\documentclass[11pt]{article}

\usepackage{geometry}
\usepackage{jhtitle,amsbsy,amsmath,amsthm,amssymb,times,graphicx}
\usepackage{probmacros}
\usepackage{natbib}

\newcommand{\vertiii}[1]{{\left\vert\kern-0.25ex\left\vert\kern-0.25ex\left\vert #1 
		\right\vert\kern-0.25ex\right\vert\kern-0.25ex\right\vert}}

\newcommand{\pcite}[1]{\citeauthor{#1}'s \citeyearpar{#1}}

\newcommand{\df}{\mathrm{d}}
\def\baro{\vskip  .2truecm\hfill \hrule height.5pt \vskip  .2truecm}
\def\barba{\vskip -.1truecm\hfill \hrule height.5pt \vskip .4truecm}

\newtheorem{theorem}{Theorem}
\newtheorem{lemma}[theorem]{Lemma}
\newtheorem{corollary}[theorem]{Corollary}
\newtheorem{remark}[theorem]{Remark}
\newtheorem{proposition}[theorem]{Proposition}

\newcommand{\X}{{\mathsf{X}}}
\newcommand{\Y}{{\mathsf{Y}}}

\author{Qian Qin and James P. Hobert \\ Department of Statistics
  \\ University of Florida} \date{April 2018} 

\begin{document}

\title{Convergence complexity analysis of Albert and Chib's algorithm
	for Bayesian probit regression}
\keywords{Drift condition, Geometric ergodicity, High dimensional inference, Large $p$ - small $n$, Markov chain Monte Carlo, Minorization condition}

\maketitle

\begin{abstract}
	The use of MCMC algorithms in high dimensional Bayesian problems has
	become routine.  This has spurred so-called convergence complexity
	analysis, the goal of which is to ascertain how the convergence rate
	of a Monte Carlo Markov chain scales with sample size,~$n$, and/or
	number of covariates,~$p$.  This article provides a thorough
	convergence complexity analysis of \pcite{albert1993bayesian} data
	augmentation algorithm for the Bayesian probit regression model.
	The main tools used in this analysis are drift and minorization
	conditions.  The usual pitfalls associated with this type of
	analysis are avoided by utilizing centered drift functions, which
	are minimized in high posterior probability regions, and by using a
	new technique to suppress high-dimensionality in the construction of
	minorization conditions.  The main result is that the geometric
	convergence rate of the underlying Markov chain is bounded below~1
	both as $n \rightarrow \infty$ (with~$p$ fixed), and as $p
	\rightarrow \infty$ (with~$n$ fixed).  Furthermore, the first
	computable bounds on the total variation distance to stationarity
	are byproducts of the asymptotic analysis.
\end{abstract}

\section{Introduction}
\label{sec:intro}

Markov chain Monte Carlo (MCMC) has become an indispensable tool in
Bayesian analysis, and it is now well known that the ability to
utilize an MCMC algorithm in a principled manner (with regard to
choosing the Monte Carlo sample size, for example) requires an
understanding of the convergence properties of the underlying Markov
chain \citep[see, e.g.,][]{fleg:hara:jone:2008}.  Taking this a step
further, in modern high dimensional problems it is also important to
understand how the convergence properties of the chain change as the
sample size, $n$, and/or number of covariates, $p$, increases.  Denote
the data (i.e., responses and covariates) by~$\mathcal{D}$.  If we
imagine~$n$ or~$p$ (or both) increasing, this leads to consideration
of a sequence of data sets, $\{\mathcal{D}_j\}$, and corresponding
sequences of posterior distributions and Monte Carlo Markov chains.  A
natural question to ask is ``What can we say about the convergence
properties of the Markov chains as $j \to \infty$?''  There is
currently a great deal of interest in questions like this in the MCMC
community \citep[see,
e.g.][]{rajaratnam2015mcmc,durmus2016high,johndrow2016inefficiency,yang2016computational,yang2017complexity}.
\citet{rajaratnam2015mcmc} call this the study of \textit{convergence
	complexity}, and we will follow their lead.

Asymptotic analysis of the convergence properties of a sequence of
Markov chains associated with increasingly large \textit{finite} state
spaces has a long history in the computer science literature, dating
back at least to \citet{sinclair1989approximate}. While the techniques developed in
computer science have been successfully applied to a few problems in
statistics \citep[see, e.g.,][]{yang2016computational}, they are generally not
applicable in situations where the state space is high-dimensional and
uncountable, which is the norm for Monte Carlo Markov chains in Bayesian
statistics.  In this paper, we employ methods based on drift and
minorization conditions to analyze the convergence complexity of one
such Monte Carlo Markov chain.

Let $\pi: \X \to [0,\infty)$ denote an intractable probability density
function (pdf), where $\X \subset \mathbb{R}^d$, and let
$\Pi(\cdot)$ denote the corresponding probability measure, i.e., for
measurable $C$, $\Pi(C) = \int_C \pi(x) \, \df x$.  
Let $K(x,\cdot)$, $x \in \X$,
denote the Markov transition function (Mtf) of an irreducible,
aperiodic, Harris recurrent Markov chain with invariant probability
measure $\Pi$.  (See \citet{meyn2012markov} for definitions.)  The
chain is called \textit{geometrically ergodic} if there exist $M :
\X \rightarrow [0,\infty)$ and $\rho \in [0,1)$ such that
\begin{equation}
\label{eq:ge}
\norm{K^m(x,\cdot) - \Pi(\cdot)}_{\mbox{\tiny{TV}}} \le M(x) \rho^m
\;\; \mbox{for all $x \in \X$ and all $m \in \mathbb{N}$} \;,
\end{equation}
where $\norm{\cdot}_{\mbox{\tiny{TV}}}$ denotes total variation norm,
and $K^m(x,\cdot)$ is the $m$-step Mtf.  The important practical
benefits of basing one's MCMC algorithm on a geometrically ergodic
Markov chain have been well-documented by, e.g.,
\citet{robe:rose:1998}, \citet{jones2001honest},
\citet{fleg:hara:jone:2008} and \citet{latu:mias:niem:2013}.  Define
the \textit{geometric convergence rate} of the chain as
\[
\rho_* = \inf \big\{ \rho \in [0,1] : \text{\eqref{eq:ge} holds for 
	some } M: \X \to [0,\infty) \big\} \;.
\]
Clearly, the chain is geometrically ergodic if and only if $\rho_* <
1$.

Establishing the convergence rate of a practically relevant Monte
Carlo Markov chain can be quite challenging.  A key tool for this
purpose has been the technique developed by
\citet{rosenthal1995minorization}, which allows for the construction
of an upper bound on $\rho_*$ using drift and minorization conditions  \citep[see also][]{meyn1994computable,roberts1999bounds,baxendale2005renewal,hairer2011yet}.
This method, which is described in detail in Section~\ref{sec:d_m},
has been used to establish the geometric ergodicity of myriad Monte
Carlo Markov chains \citep[see,
e.g.,][]{fort2003geometric,marchev2004geometric,roy2010monte,vats2017geometric}.
Since methods based on drift and minorization (hereafter, d\&m) are still the
most (and arguably the only) reliable tools for bounding~$\rho_*$ for
practically relevant Monte Carlo Markov chains on uncountable state
spaces, it is important to know whether they remain useful in the
context of convergence complexity analysis.  Unfortunately, it turns
out that most of the upper bounds on $\rho_*$ that have been produced
using techniques based on d\&m converge to~1 (often exponentially
fast), becoming trivial, as $n$ and/or $p$ grow \citep[see,
e.g.,][]{rajaratnam2015mcmc}.  One example of this is
\pcite{roy2007convergence} analysis of \pcite{albert1993bayesian} data
augmentation algorithm for the Bayesian probit model, which
establishes geometric ergodicity of the underlying Markov chain, but
also leads to an upper bound on $\rho_*$ that converges to~1 as $n
\rightarrow \infty$.

There are, of course, many possible explanations for why the
d\&m-based upper bounds on $\rho_*$ converge to~1.  It could be that
the associated Monte Carlo Markov chains actually have poor asymptotic
properties, or, if not, perhaps d\&m-based
methods are simply not up to the more delicate task of convergence
complexity analysis.  We show that, in the case of
\pcite{albert1993bayesian} chain, neither of these potential
explanations is correct.  Indeed, our careful d\&m analysis of
\pcite{albert1993bayesian} chain (hereafter, A\&C's chain) leads to
upper bounds on~$\rho_*$ that are bounded away from 1 in both the
large~$n$, small~$p$ case, and the large~$p$, small~$n$ case.  We
believe that this is the first successful convergence complexity
analysis of a practically relevant Monte Carlo Markov chain using
d\&m.
The key ideas we use to establish our results
include ``centering'' the drift function to a region in the state
space that the chain visits frequently, and suppressing
high-dimensionality in the construction of minorization conditions.
In particular, for two-block Gibbs chains, we introduce a technique for
constructing asymptotically stable minorization conditions that is
based on the well known fact that the two marginal Markov chains and
the joint chain all share the same geometric convergence rate.

Recently, \citet{yang2017complexity} used a modified version of
\pcite{rosenthal1995minorization} technique to successfully analyze
the convergence complexity of a Gibbs sampler for a simple Bayesian
linear mixed model.  We note that, because one of the variance
components in their model is assumed known, it is actually
straightforward to sample directly from the posterior distribution
using a univariate rejection sampler \citep[Section 3.9]{jone:2001}.
Thus, while \pcite{yang2017complexity} results are impressive, and
their methods may suggest a way forward, the Monte Carlo Markov chain
that they analyzed is not practically relevant.

Before describing our results for A\&C's chain, we introduce an
alternative definition of geometric ergodicity.  Let $L^2(\Pi)$ denote
the set of signed measures $\mu$ that are absolutely continuous with
respect to~$\Pi$, and satisfy $\int_\X ( \df\mu/\df\Pi )^2
\df \Pi < \infty$.  As in \citet{roberts1997geometric}, we say that
the Markov chain with Mtf~$K$ is $L^2$-\textit{geometrically ergodic} if there exists
$\rho<1$ such that for each probability measure $\nu \in L^2(\Pi)$,
there exists a constant $M_{\nu} < \infty$ such that
\begin{equation*}
\norm{\nu K^m (\cdot) - \Pi (\cdot)}_{\mbox{\tiny{TV}}} \le
M_{\nu} \rho^m \;\; \mbox{for all $m \in \mathbb{N}$} \;,
\end{equation*}
where $\nu K^m(\cdot) = \int_\X K^m(x,\cdot) \, \nu(\df x)$.  We
define the $L^2$-\textit{geometric convergence rate}, $\rho_{**}$, to
be the infimum of all $\rho \in [0,1]$ that satisfy this definition.
Not surprisingly, $\rho_{*}$ and $\rho_{**}$ are closely related
\citep[see, e.g.,][]{roberts2001geometric}.

We now provide an overview of our results for
\pcite{albert1993bayesian} Markov chain, starting with a brief
description of their algorithm.  Let $\{X_i\}_{i=1}^n$ be a set of
$p$-dimensional covariates, and let $\{Y_i\}_{i=1}^n$ be a
corresponding sequence of binary random variables such that $Y_i|X_i,B
\sim \mbox{Bernoulli}(\Phi(X_i^T B))$ independently, where $B$ is a $p
\times 1$ vector of unknown regression coefficients, and~$\Phi$ is the
standard normal distribution function.  Consider a Bayesian analysis
using a prior density for $B$ given by
\begin{equation} \label{eq:prior}
\omega(\beta) \propto \exp \Big \{ -\frac{1}{2} (\beta-v)^TQ(\beta-v)
\Big\} \;,
\end{equation}
where $v \in \mathbb{R}^p$, and $Q \in \mathbb{R}^{p \times p}$ is
either a positive definite matrix (proper Gaussian prior), or a zero
matrix (flat improper prior).  Assume for now that the posterior is
proper.  (Propriety under a flat prior is discussed in
Section~\ref{sec:bayesprobitac}.)  As usual, let $X$ denote the $n
\times p$ matrix whose $i$th row is $X_i^T$, and let $Y = (Y_1 \; Y_2
\cdots Y_n)^T$ denote the vector of responses.  The intractable
posterior density is given by
\begin{equation}
\label{eq:post_I}
\pi_{B|Y,X}(\beta|Y,X) \propto \left\{ \prod_{i=1}^n \left(\Phi(X_i^T
\beta)\right)^{Y_i} \left(1-\Phi(X_i^T \beta)\right)^{1-Y_i} \right\}
\omega(\beta) \;.
\end{equation}
The standard method for exploring~\eqref{eq:post_I} is the classical
data augmentation algorithm of~\citet{albert1993bayesian}, which
simulates a Harris ergodic (irreducible, aperiodic, and Harris
recurrent) Markov chain, $\{B_m\}_{m=0}^\infty$, that has invariant
density $\pi_{B|Y,X}$.  In order to state the algorithm, we require a
bit of notation.  For $\theta \in \mathbb{R}$, $\sigma > 0$, and $i
\in \{0,1\}$, let $\mbox{TN}(\theta, \sigma^2; i)$ denote the
$\mbox{N}(\theta,\sigma^2)$ distribution truncated to $(0,\infty)$ if
$i=1$, and to $(-\infty,0)$ if $i=0$.  The matrix $\Sigma := X^TX + Q$
is necessarily non-singular because of propriety.  If the current
state of A\&C's chain is $B_m = \beta$, then the new state,
$B_{m+1}$, is drawn using the following two steps:

\baro \vspace*{2mm}
\noindent {\rm Iteration $m+1$ of the data augmentation algorithm:}
\begin{enumerate}
	\item Draw $\{Z_i\}_{i=1}^n$ independently with $Z_i \sim
	\mbox{TN}(X_i^T \beta, 1; Y_i)$, and let $Z = (Z_1 \; Z_2 \cdots
	Z_n)^T$.
	\item Draw
	$
	B_{m+1} \sim \mbox{N}_p \Big( \Sigma^{-1} \big(X^TZ + Qv \big),
	\Sigma^{-1} \Big) \;.
	$
\end{enumerate}

\barba
\bigskip

The convergence rate of A\&C's chain has been studied by several
authors.  \citet{roy2007convergence} proved that when $Q=0$, the chain
is always geometrically ergodic.  (Again, we're assuming posterior
propriety.)  A similar result for proper normal priors was established
by \citet{chakraborty2016convergence}.  Both results were established
using a technique that does not require construction of a minorization
condition \citep[see][Lemma 15.2.8]{meyn2012markov}, and,
consequently, does not yield an explicit upper bound on $\rho_*$.
Thus, neither paper addresses the issue of convergence complexity.
However, in Section~\ref{sec:rp2} we prove that
\pcite{roy2007convergence} drift function \textit{cannot} be used to
construct an upper bound on $\rho_*$ that is bounded away from~1 as $n
\rightarrow \infty$.  \citet{johndrow2016inefficiency} recently
established a convergence complexity result for the intercept only
version of the model ($p=1$ and $X_i = 1$ for $i = 1,2,\dots,n$) with
a proper (univariate) normal prior, under the assumption that
\textit{all} the responses are successes ($Y_i = 1$ for $i =
1,2,\dots,n$).  Their results, which are based on Cheeger's
inequality, imply that $\rho_{**} \to 1$ as $n \to \infty$, indicating
that the algorithm is inefficient for large samples.

The results established herein provide a much more complete picture of
the convergence behavior of A\&C's chain.  Three different regimes are
considered: (i) fixed $n$ and $p$, (ii) large $n$, small $p$, and
(iii) large $p$, small $n$.  Our analysis is based on two different
drift functions that are both appropriately centered (at the posterior
mode).  One of the two drift functions is designed for regime (ii), and
the other for regime (iii).  We establish d\&m conditions for both drift
functions, and these are used in conjunction with
\pcite{rosenthal1995minorization} result to construct two explicit
upper bounds on $\rho_*$.  They are also used to construct two
computable upper bounds on the total variation distance to
stationarity (as in \eqref{eq:ge}), which improves upon the analyses
of \citet{roy2007convergence} and \citet{chakraborty2016convergence}.

The goal in regime (ii) is to study the asymptotic behavior of the
geometric convergence rate as $n \rightarrow \infty$, when $p$ is
fixed.  To this end, we consider a sequence of data sets,
$\mathcal{D}_n := \{(X_i,Y_i)\}_{i=1}^n$.  So, each time $n$ increases
by 1, we are are given a new $p \times 1$ covariate vector, $X_i$, and
a corresponding binary response, $Y_i$.  To facilitate the asymptotic
study, we assume that the $(X_i,Y_i)$ pairs are generated according to
a \textit{random} mechanism that is governed by very weak assumptions
(that are consistent with the probit regression model).  We show that
there exists a constant $\rho<1$ such that, almost surely, $\limsup_{n
	\to \infty} \rho_*(\mathcal{D}_n) \leq \rho$.  Apart from this
general result, we are also able to show that, in the intercept only
model considered by \citet{johndrow2016inefficiency}, the A\&C chain
is actually quite well behaved as long as the proportion of successes
is bounded away from 0 and 1.  To be specific, let
$\{Y_i\}_{i=1}^\infty$ denote a \textit{fixed} sequence of binary
responses, and let $\hat{p}_n = n^{-1} \sum_{i=1}^n Y_i$.  Our results
imply that, as long as $0 < \liminf_{n \to \infty} \hat{p}_n \le
\limsup_{n \to \infty} \hat{p}_n < 1$, there exists $\rho<1$ such that
both $\rho_*(\mathcal{D}_n)$ and $\rho_{**}(\mathcal{D}_n)$ are
eventually bounded above by~$\rho \hspace*{.3mm}$, and there is a
closed form expression for~$\rho \hspace*{.3mm}$.

In regime (iii), $n$ is fixed and $p \rightarrow \infty$.  There are
several important differences between regimes (ii) and (iii).  First, in
regime (ii), since $p$ is fixed, only a single prior distribution need
be considered.  In contrast, when $p \rightarrow \infty$, we must
specify a sequence of priors, $\{(Q_p,v_p)\}_{p=1}^\infty$, where
$v_p$ is a $p \times 1$ vector, and $Q_p$ is a $p \times p$ positive
definite matrix.  (When $p > n$, a positive definite $Q_p$ is required
for posterior propriety.)  Also, in regime (iii), there is a fixed
vector of responses (of length $n$), and it is somewhat unnatural to
consider the new columns of $X$ to be random.  Let $\{ \mathcal{D}_p \} :=
\{(v_p,Q_p,X_{n \times p},Y)\}$ denote a fixed sequence of priors and data sets,
where $Y$ is a fixed $n \times 1$ vector of responses, and $X_{n \times p}$ is an
$n \times p$ matrix.  We show that, under a natural regularity
condition on $X_{n \times p} Q_p^{-1} X_{n \times p}^T$, there exists a $\rho<1$ such that
$\rho_*(\mathcal{D}_p) \leq \rho$ for all $p$.

The remainder of the paper is laid out as follows.  In
Section~\ref{sec:d_m}, we formally introduce the concept of
\textit{stable} d\&m conditions, and describe techniques that we employ for
constructing such.  The centered drift functions that are used in our
analysis of A\&C's chain are described in
Section~\ref{sec:bayesprobitac}.  In Section~\ref{sec:rp1}, we provide
results for A\&C's chain in the case where $n$ and $p$ are both fixed.
Two sets of d\&m conditions are established, and corresponding exact
total variation bounds on the distance to stationarity are provided.
The heart of the paper is Section~\ref{sec:rp2} where it is shown that
the geometric convergence rate of A\&C's chain is bounded away from 1
as $n \rightarrow \infty$ for fixed $p$, and as $p \rightarrow \infty$
for fixed $n$.  A good deal of technical material is relegated to the Appendix.

\section{Asymptotically Stable Drift and Minorization}
\label{sec:d_m}

Let $\mathsf{X}$ be a set equipped with a countably generated
$\sigma$-algebra ${\cal B}(\mathsf{X})$.  Suppose that $K: \mathsf{X}
\times {\cal B}(\mathsf{X}) \rightarrow [0,1]$ is an Mtf with
invariant probability measure $\Pi(\cdot)$, so that $\Pi(C) = \int_\X
K(x,C) \Pi(\df x)$ for all $C \in \cal B(\X)$.  Assume that the
corresponding Markov chain is Harris ergodic.  Recall the definitions
of geometric ergodicity and geometric convergence rate from
Section~\ref{sec:intro}.  The following result has proven extremely
useful for establishing geometric ergodicity in the context of Monte
Carlo Markov chains used to study complex Bayesian posterior
distributions.

\begin{theorem}[\citet{rosenthal1995minorization}]
	\label{thm:rosen}
	Suppose that $K(x,\cdot)$ satisfies the drift condition
	\begin{equation}
	\label{ine:drift}
	\int_{\X} V(x') K(x,\df x') \leq \lambda V(x) + L, \quad x
	\in \X
	\end{equation}
	for some $V: \X \to [0,\infty)$, $\lambda < 1$ and $L <
	\infty$. Suppose that it also satisfies the minorization condition
	\begin{equation}
	\label{ine: min}
	K(x,\cdot) \geq \varepsilon Q(\cdot) \quad \text{ whenever } V(x) < d
	\end{equation}
	for some $\varepsilon > 0$, probability measure $Q(\cdot)$ on~$\X$,
	and $d > 2L/(1-\lambda)$.  Then assuming the chain is started
	according to the probability measure $\nu(\cdot)$, for any $0 < r <
	1$, we have
	\begin{align*}
	\|\nu &K^m(\cdot) - \Pi(\cdot)\|_{\mbox{\tiny{TV}}} \leq ( 1 -
	\varepsilon)^{rm} + \\& \left( 1 + \frac{L}{1-\lambda} + \int_\X
	V(x) \nu(\df x) \right) \left[ \left( \frac{1+2L+\lambda d}{1 + d}
	\right)^{1-r} \left\{1 + 2(\lambda d + L) \right\}^r \right]^m \;.
	\end{align*}
\end{theorem}

The function $V$ is called the drift (or Lyapunov) function, and $\{x
\in \X: V(x) < d\}$ is called the small set associated with~$V$.  We
will refer to $\varepsilon$ as the minorization number.  Manipulation
of the total variation bound in Theorem~\ref{thm:rosen} leads to
\[
\norm{\nu K^m(\cdot) - \Pi(\cdot)}_{\mbox{\tiny{TV}}} \leq \bigg( 2 +
\frac{L}{1-\lambda} + \int_\X V(x) \nu(\df x) \bigg) \hat{\rho}^m \;,
\]
where 
\begin{equation}
\label{eq:ros_bound}
\hat{\rho} := (1-\varepsilon)^r \vee \left( \frac{1+2L+\lambda d}{1
	+ d} \right)^{1-r} \big\{ 1 + 2(\lambda d + L) \big\}^r \,,
\end{equation}
and $r \in (0,1)$ is arbitrary.   Then, $\hat{\rho}$
is an upper bound on the geometric convergence rate~$\rho_*$. It's easy to verify
that when $\lambda < 1$, $L < \infty$, and $\varepsilon > 0$, there
exists $r \in (0,1)$ such that $\hat{\rho} < 1$. 

The bound $\hat{\rho}$ has a reputation for being too conservative.
This is partly due to the fact that there are toy examples where the
true $\rho_*$ is known, and $\hat{\rho}$ is quite far off \citep[see,
e.g.,][]{rosenthal1995minorization}.  There also exist myriad
analyses of practical Monte Carlo Markov chains where the d\&m
conditions \eqref{ine:drift} and \eqref{ine: min} have been
established (proving that the underlying chain is indeed geometrically
ergodic), but the total variation bound of Theorem~\ref{thm:rosen} is
useless because $\hat{\rho}$ is so near unity.  Of course, the quality
of the bound $\hat{\rho}$ depends on the choice of drift function, and
the sharpness of~\eqref{ine:drift} and~\eqref{ine: min}.  Our results
for the A\&C chain suggest that poorly chosen drift functions and/or
loose inequalities in the d\&m conditions are to blame for (at least)
some of the unsuccessful applications of Theorem~\ref{thm:rosen}.  We
now introduce the concept of asymptotically stable d\&m.

Consider a sequence of geometrically ergodic Markov chains,
$\{\Psi^{(j)}\}_{j=1}^\infty$, with corresponding geometric
convergence rates given by $\rho_*^{(j)}$.  (In practice,~$j$ is
usually the sample size,~$n$, or number of covariates,~$p$.) We are
interested in the asymptotic behavior of the rate sequence.  For
example, we might want to know if it is bounded away from~1.  Suppose
that for each chain, $\Psi^{(j)}$, we have d\&m conditions defined
through $\lambda^{(j)}$, $L^{(j)}$, and $\varepsilon^{(j)}$, and thus
an upper bound on the convergence rate, $\hat{\rho}^{(j)} \in
[\rho_*^{(j)},1)$.  The following simple result (whose proof is left
to the reader) provides conditions under which these upper bounds
are \textit{unstable}, that is, $\limsup_{j \to \infty}
\hat{\rho}^{(j)} = 1$.

\begin{proposition}
	\label{prop:stability}
	Suppose that there exists a subsequence of
	$\{\Psi^{(j)}\}_{j=0}^{\infty}$, call it
	$\{\Psi^{(j_l)}\}_{l=0}^{\infty}$, that satisfies one or more
	of the following three conditions: (i) $\lambda^{(j_l)} \to
	1$, (ii) $L^{(j_l)} \to \infty$, while~$\varepsilon^{(j_l)}$
	is bounded away from~$1$ and~$\lambda^{(j_l)}$ is bounded away
	from~$0$, (iii) $\varepsilon^{(j_l)} \to 0$.  Then the
	corresponding subsequence of the upper bounds,
	$\hat{\rho}^{(j_l)}$, converges to~1.
\end{proposition}

If one (or more) of the conditions in Proposition~\ref{prop:stability}
holds for some subsequence of $\{\Psi^{(j)}\}_{j=0}^{\infty}$, then we
say that the d\&m conditions are (asymptotically) unstable (in~$j$).
On the other hand, if $(\mbox{i}')$~$\lambda^{(j)}$ is bounded away
from~$1$, $(\mbox{ii}')$~$L^{(j)}$ is bounded above, and
$(\mbox{iii}')$~$\varepsilon^{(j)}$ is bounded away from~$0$, then the
sequence $\hat{\rho}^{(j)}$ can be bounded away from~$1$, thus giving
an asymptotically nontrivial upper bound on~$\rho_*^{(j)}$.  We say
that the drift conditions are stable if $(\mbox{i}')$ and
$(\mbox{ii}')$ hold, and likewise, that the minorization conditions
are stable if $(\mbox{iii}')$ holds.

Before moving on to describe the techniques that we use to develop
stable d\&m for the A\&C chain, we note that elementary linear
transformations of the drift function can affect the quality
of~$\hat{\rho}$, and even stability.  It's easy to show that, while
multiplying the drift function by a scale factor will affect~$L$, it
will not affect the quality of the minorization inequality~\eqref{ine:
	min} in any non-trivial way.  Subtracting a positive number from~$V$
(while preserving its non-negativity) will, on the other hand, always
lead to an improved bound~$\hat{\rho}$.  Needless to say, we will only
deal with instability that cannot be prevented by these trivial
transformations.  In particular, throughout the article, we consider
only drift functions whose infimums are~$0$, that is, we make the
elementary transformation $V(x) \mapsto V(x) - \inf_{x' \in \X}
V(x')$.

To obtain stable d\&m for the A\&C chain, we will exploit the notion
of ``centered" drift functions. Theorem~\ref{thm:rosen} is based on a
coupling construction, in which two copies of the Markov chain
coalesce with probability~$\varepsilon$ each time they (both) enter
the small set. The total variation distance between the two chains at
time~$m$ is then bounded above by~$1$ minus the probability of
coalescence in~$m$ iterations. Thus, loosely speaking, we want the
chains to visit the small set as often as possible, without making the
small set too large. (Larger small sets usually result in
smaller~$\varepsilon$, as indicated by~\eqref{ine: min}.)  So, it
makes sense to use a drift function that is centered in the sense
that it takes small values in the part of the state space where the
chain spends the bulk of its time.  Of course, if the chain is well
suited to the problem, then it should linger in the high posterior
probability regions of~$\X$.

The idea of centering is not new, and has been employed without
emphasis by many authors.  In this article, we illustrate the
importance of centering to stable d\&m, especially when~$n$ is large.
Indeed, in Section~\ref{sec:rp2} it is shown that, for A\&C's chain,
in the large~$n$, small~$p$ regime, the un-centered drift function
employed by \citet{roy2007convergence} cannot possibly lead to stable
d\&m, while a centered version of the same drift function does.  The
intuition behind these results is as follows.  By~\eqref{ine:drift},
$d > 2L/(1-\lambda) \geq 2\Pi V$, where $\Pi V := \int_\X V(x)
\Pi(dx)$. Hence, $\Pi V$ controls the volume of the small set.  In
Bayesian models, as~$n$ increases, the posterior is likely to
concentrate around a single point in the state space.  Consider a
sequence of posterior distributions and drift functions,
$\{(\Pi^{(n)}, V^{(n)})\}$, such that $\Pi^{(n)}$ concentrates around
a point $x_0$. Heuristically, we expect $\Pi^{(n)} V^{(n)}$ to be
close to $V^{(n)}(x_0)$ for large~$n$.  Therefore, when $n$ is large,
if the drift functions are minimized at or near $x_0$, then we will
have a better chance of controlling the volumes of the small sets, and
bounding the minorization numbers away from 0.

Another technique we use to achieve stable d\&m for the A\&C chain is
a dimension reduction trick that is designed specifically for
two-block Gibbs samplers and data augmentation algorithms. We begin by
describing a common difficulty encountered in the convergence analysis
of such algorithms.  Suppose that $\X \subset \mathbb{R}^p$, and $K(x,
\cdot)$ is associated with a Markov transition density (Mtd) of the
form
\begin{equation} 
\label{eq:da}
k(x,x') = \int_{\mathbb{R}^n} s(x'|z) h(z|x) \, \df z \;,
\end{equation}
where $n$ is the sample size, $z = (z_1 \,z_2 \cdots z_n)^T \in
\mathbb{R}^n$ is a vector of latent data, and $s: \X \times
\mathbb{R}^n \rightarrow [0,\infty)$ and $h: \mathbb{R}^n \times \X
\rightarrow [0,\infty)$ are conditional densities associated with
some joint density on $\X \times \mathbb{R}^n$. Assume that
$h(z|x)$ can be factored as follows:
\[
h(z|x) = \prod_{i=1}^{n} h_i(z_i|x) \;,
\]
where, for each $i$, $h_i: \mathbb{R} \times \X \rightarrow
[0,\infty)$ is a univariate (conditional) pdf.  Usually, $s(\cdot|z)$
and $h_i(\cdot|x)$ are tractable, but there is no closed form for
$k(x,x')$.  However, there is a well-known argument for establishing
a minorization condition in this case. Suppose that whenever $V(x) <
d$,
\[
h_i(t|x) > \epsilon_i \nu_i(t) \;, \quad i=1,2,\dots,n
\]
where $\epsilon_i > 0$ and $\nu_i: \mathbb{R} \to [0,\infty)$ is a
pdf.  Then, whenever $V(x)<d$, we have
\begin{equation}
\label{unstablemin}
k(x,x') > \left( \prod_{i=1}^{n} \epsilon_i \right)
\int_{\mathbb{R}^n} s(x'|z) \prod_{i=1}^{n} \nu_i(z_i) \, \df z \;.
\end{equation}
Since $\int_{\mathbb{R}^n} s(x'|z) \prod_{i=1}^{n} \nu_i(z_i) \, \df
z$ is a pdf on $\X$, \eqref{unstablemin} gives a minorization
condition with $\varepsilon = \prod_{i=1}^{n} \epsilon_i$.
Unfortunately, this quantity will almost always converge to 0 as $n
\to \infty$.  Consequently, if our sequence of Markov chains are
indexed by $n$, then we have unstable minorization.  This problem is
well known \citep[see, e.g.,][]{rajaratnam2015mcmc}.

The instability of the minorization described above is due to the fact
that the dimension of $z$ is growing with $n$.  However, it is often
the case that $s(x|z)$ depends on~$z$ only through $f(z)$, where $f:
\mathbb{R}^n \to \Y$ is a function into some fixed space $\Y$, say $\Y
\subset \mathbb{R}^q$, where $q$ does not depend on $n$.  Then
integrating along $f(z) = \gamma$ in~\eqref{eq:da} yields
\begin{equation}
\label{eq:alt_rep}
k(x,x') = \int_{\Y} \tilde{s} \left(x'|\gamma \right) \tilde{h}
\big(\gamma|x \big) \, \df \gamma \;,
\end{equation}
where $\tilde{s}(x'|f(z)) = s(x'|z)$, and
\[
\int_C \tilde{h}(\gamma|x) \, \df \gamma = \int_{\{z: f(z)\in C\}}
h(z|x) \, \df z
\]
for all $x \in \X$ and any measurable $C \subset \Y$.  Note that this
new representation of $k(x,x')$ no longer contains~$n$ explicitly, and
the high dimensionality problem for~$z$ is resolved.  However, we now
have a new problem.  Namely, $\tilde{h}(\gamma|x)$ is likely to be
quite intractable.  Fortunately, the following result provides a way
to circumvent this difficulty.

\begin{proposition}
	\label{prop:bad_ref}
	Assume that we have a drift condition for $k(x,x')$, i.e.,
	\begin{equation} \label{driftori}
	\int_{\X} V(x') k(x,x') \, \df x' \leq \lambda V(x) + L \;,
	\end{equation}
	where $V:\X \to [0,\infty)$, $\lambda \in [0,1)$, and $L$ is finite.  Assume further that
	$k(x,x')$ can be written in the form~\eqref{eq:alt_rep}.  Define a
	Mtd $\tilde{k}: \Y \times \Y \rightarrow [0,\infty)$ as follows:
	\[
	\tilde{k}(\gamma, \gamma') = \int_{\X} \tilde{h}(\gamma'|x)
	\tilde{s}(x|\gamma) \, \df x \,.
	\]
	If $\tilde{V}(\gamma) = \int_{\X} V(x) \tilde{s}(x|\gamma) \, \df x + c$  is finite and non-negative for all
	$\gamma \in \Y$, where~$c$ is some constant, then the following drift condition holds for~$\tilde{k}$,
	\begin{equation} \label{driftflipped}
	\int_{\Y} \tilde{V}(\gamma') \tilde{k}(\gamma,\gamma') \, \df \gamma' < \lambda \tilde{V}(\gamma) + \tilde{L} \;,
	\end{equation}
	where $\tilde{L} = L + c(1-\lambda)$.
\end{proposition}

\begin{proof}
	Our assumptions imply that
	\begin{equation} \nonumber
	\int_{\X} V(x') \int_{\Y} \tilde{s}(x'|\gamma' ) \tilde{h}(\gamma'|x)
	\, \df \gamma' \, \df x' \leq \lambda V(x) + L \;.
	\end{equation}
	Multiplying both sides of the inequality by $\tilde{s}(x|\gamma)$,
	and integrating out~$x$ yields the result.
\end{proof}
\begin{remark}
	Note that if we set $c \leq 0$ (while preserving the non-negativity of~$\tilde{V}$) in the above proposition, then~\eqref{driftflipped} is stable whenever the original drift~\eqref{driftori} is stable.
\end{remark}

As we now explain, Proposition~\ref{prop:bad_ref} has important
implications.  Indeed, it is well known that the Markov chains driven
by $k$ and $\tilde{k}$ (which we call the ``flipped'' version of $k$)
share the same geometric convergence rate \citep[see,
e.g.,][]{roberts2001markov,diaconis2008gibbs}.  Thus, we can analyze~$k$ indirectly through the flipped chain.  Now, as mentioned
above, $\tilde{s}(x|f(z)) = s(x|z)$ is often tractable.  Suppose that
there exists some $\tilde{\varepsilon} > 0$ and pdf $\tilde{\nu}: \X
\to [0, \infty)$ such that $\tilde{s}(x|\gamma) \geq
\tilde{\varepsilon} \, \tilde{\nu}(x)$ when $\tilde{V}(\gamma) < \tilde{d}$, where $\tilde{d} > 2\tilde{L}/(1-\lambda)$.
Then we have the following minorization condition for the flipped
chain:
\[
\tilde{k}(\gamma, \gamma') \geq \tilde{\varepsilon}
\int_{\X} \tilde{h}(\gamma'|x) \tilde{\nu}(x) \, \df x \quad
\text{ whenever } \tilde{V}(\gamma) < \tilde{d} \;,
\]
%
which is stable as long as~$\tilde{\varepsilon}$ is bounded away
from~$0$ as $n \rightarrow \infty$. This, along with~\eqref{driftflipped}, allows us to construct potentially stable bounds on~$\rho_{*}^{(n)}$ for the flipped chains, and thus for the original chains on~$\X$ as well.  This is exactly how
we analyze A\&C's chain in the large~$n$, small~$p$ regime.  It turns
out that the flipped chain argument can also be used to establish d\&m
conditions that are stable in~$p$, and we will exploit this in our
analysis of A\&C's chain in the large~$p$, small~$n$ regime.

We end this section with a result that allows us to use information
about a flipped chain to get total variation bounds for the original.
The following result follows immediately from
Proposition~\ref{prop:tv_flip}, which is stated and proven in
Subsection \ref{app:tv_flip} of the Appendix.

\begin{corollary}
	\label{cor:tv_flip}
	Suppose we have d\&m conditions for a flipped
	chain (which is driven by $\tilde{k}$).  Let $\lambda$, $\tilde{L}$, and
	$\tilde{\varepsilon}$ denote the drift and minorization parameters,
	and let $\hat{\rho}_f$ denote the corresponding bound on the
	geometric convergence rate obtained through Theorem~\ref{thm:rosen}. 
	Then, if the original chain is started at  $x \in \X$,
	we have, for $m \ge 1$,
	\[
	\norm{K^m(x, \cdot) - \Pi(\cdot)}_{\mbox{\tiny{TV}}} \leq \bigg( 2 +
	\frac{\tilde{L}}{1-\lambda} + \int_{\Y} \tilde{V}(\gamma) \tilde{h}(\gamma|x) \, \df \gamma \bigg) \hat{\rho}_f^{m-1} \;.
	\]
\end{corollary}

In the next section, we begin our analysis of the A\&C chain.

\section{Albert and Chib's Markov Chain and the Centered Drift Functions}
\label{sec:bayesprobitac}

\subsection{Basics}
\label{ssec:basics}

Let $\{X_i\}_{i=1}^n$, $\{Y_i\}_{i=1}^n$, and $B \in \mathbb{R}^p$ be defined as in the Introduction, so that 
$Y_i|X_i,B \sim \mbox{Bernoulli}(\Phi(X_i^T B))$
independently.
Suppose that, having observed the data,
$\mathcal{D} := \{(X_i,Y_i)\}_{i=1}^n$, we wish to perform a Bayesian
analysis using a prior density for $B$ given by~\eqref{eq:prior}. Recall that~$X$ and~$Y$ denote, respectively, the design matrix and vector of responses.  The
posterior density~\eqref{eq:post_I} is proper precisely when
\[
\int_{\mathbb{R}^p} \prod_{i=1}^n \left(\Phi(X_i^T
\beta)\right)^{Y_i} \left(1-\Phi(X_i^T \beta)\right)^{1-Y_i}
\omega(\beta) \, d\beta < \infty \;.
\]
When~$Q$ is positive definite, $\omega(\beta)$ is a proper normal
density, and the posterior is automatically proper.  If $Q = 0$, then
propriety is not guaranteed.  Define $X_*$ as the $n \times p$ matrix
whose $i$th row is $-X_i^T$ if $Y_i=1$, and $X_i^T$ if $Y_i = 0$.
\citet{chen2000propriety} proved that when the prior is flat, i.e.,
$Q=0$, the following two conditions are necessary and sufficient for
posterior propriety:
\begin{enumerate}
	\item[($C1$)] $X$ has full column rank;
	\item[($C2$)] There exists a vector $a=(a_1 \, a_2 \cdots a_n)^T \in
	\mathbb{R}^n$ such that $a_i > 0$ for all~$i$, and 
	$X_*^T a = 0$.
\end{enumerate}
Until further notice, we will assume that
the posterior is proper.

A\&C's algorithm to draw from the intractable posterior is based on
the following latent data model.  Given~$X$ and~$B$, let
$\{(Y_i,Z_i)\}_{i=1}^n$ be a sequence of independent random vectors
such that
\begin{align*}
& Y_i|Z_i, X, B \text{ is a point mass at } 1_{\mathbb{R}_+}(Z_i) \,
\\ & Z_i|X,B \sim N(X_i^TB, 1) \;.
\end{align*}
Clearly, under this hierarchical structure, $Y_i|X_i,B \sim
\mbox{Bernoulli}(\Phi(X_i^T B))$, which is consistent with the
original model.  Thus, if we let $\pi_{B,Z|Y,X}(\beta,z|Y,X)$ denote
the corresponding (augmented) posterior density (where $Z= (Z_1 \, Z_2
\cdots Z_n)^T$), then it's clear that
\[
\int_{\mathbb{R}^n} \pi_{B,Z|Y,X}(\beta,z|Y,X) \, \mbox{d$z$} =
\pi_{B|Y,X}(\beta|Y,X) \,,
\]
which is the target posterior from~\eqref{eq:post_I}.  Albert and Chib's
algorithm is simply a two-variable Gibbs sampler based on
$\pi_{B,Z|Y,X}(\beta,Z|Y,X)$.  Indeed, the Mtd, $k_{\mbox{\scriptsize{AC}}}:\mathbb{R}^p \times
\mathbb{R}^p \rightarrow \mathbb{R}_+$, is defined as
\begin{align*}
k_{\mbox{\scriptsize{AC}}}(\beta,\beta') &:=
k_{\mbox{\scriptsize{AC}}}(\beta,\beta';Y,X) \\
&= \int_{\mathbb{R}^n}
\pi_{B|Z,Y,X}(\beta'|z,Y,X) \pi_{Z|B,Y,X}(z|\beta,Y,X) \, \mbox{d$z$}
\; .
\end{align*}
As pointed out by \citet{albert1993bayesian},
\[
B|Z,Y,X \sim \mbox{N}_p \big( \Sigma^{-1} \big(X^TZ + Qv \big),
\Sigma^{-1} \big) \;,
\]
where, again, $\Sigma = X^TX + Q$.  Moreover, the density
$\pi_{Z|B,Y,X}(z|\beta,Y,X)$ is a product of $n$ univariate densities,
where
\[
Z_i|B,Y,X \sim \mbox{TN}(X_i^TB,1;Y_i).
\]
Obviously, these are the conditional densities that appear in the
algorithm described in the Introduction.

\subsection{A centered drift function}
\label{ssec:center}

\citet{roy2007convergence} and \citet{chakraborty2016convergence} both
used the drift function $V_0(\beta) = \norm{\Sigma^{1/2} \beta}^2$.
While this drift function is certainly amenable to analysis, it is not
``centered'' in any practical sense.  Indeed, $V_0(\beta)$ takes on
its minimum when $\beta=0$, but, in general, there is no reason to
expect A\&C's chain to make frequent visits to the vicinity of
the origin.  This heuristic is borne out by the result in
Section~\ref{sec:rp2} showing that $V_0$ \textit{cannot} lead to
stable d\&m in the large $n$, small $p$ regime.  As an alternative to
$V_0(\beta)$, we consider drift functions of the form
\begin{equation}
\label{eq:generic}
V(\beta) = \norm{M (\beta - \beta^*)}^2 \;,
\end{equation}
where $M=M(X,Y)$ is a matrix with $p$ columns, and
$\beta^*=\beta^*(X,Y)$ is a point in $\mathbb{R}^p$ that is
``attractive'' to A\&C's chain.  A candidate for $\beta^*$
would be the posterior mode~$\hat{B}$, which uniquely exists because of the well-known fact that
the posterior density $\pi_{B|Y,X}$ is log-concave. Setting $\beta^* = \hat{B}$ is, of course, not the only 
viable centering scheme, and any~$\beta^*$ in a close vicinity of~$\hat{B}$ would be equally effective. However, the following proposition shows that the posterior mode has a nice feature that will be exploited in the sequel.

\begin{proposition}
	\label{fixedmle}
	The posterior mode,~$\hat{B}$, satisfies the following equation,
	\begin{equation}
	\label{eq:fixed point}
	\int_{\mathbb{R}^p } \beta \, k_{\mbox{\scriptsize{AC}}}(\hat{B},\beta)
	\, \df \beta = \hat{B} \;.
	\end{equation}
\end{proposition}

\begin{proof}
	$\hat{B}$ is the
	solution to the following equation,
	\[
	\frac{d}{d\beta} \left( \log\omega(\beta) + \log
	\pi_{Y|B,X}(Y|\beta,X) \right) = 0 \;.
	\]
	This implies that
	\begin{equation}
	\label{eq:fixedeqn2}
	\sum_{i=1}^{n} \left( \frac{\phi(X_i^T \hat{B})}{\Phi(X_i^T \hat{B})}
	1_{\{1\}}(Y_i) - \frac{\phi(X_i^T \hat{B})}{1 - \Phi(X_i^T \hat{B})}
	1_{\{0\}}(Y_i) \right) \, X_i - (Q \hat{B} - Qv) = 0 \;,
	\end{equation}
	where $\phi(\cdot)$ is the pdf of the standard normal distribution.
	On the other hand, it follows from \eqref{eq:evtns} in Subsection~\ref{app:tn} of the Appendix that
	\[
	\mathbb{E}(Z_i | B = \hat{B}, Y, X) = X_i^T \hat{B} +
	\frac{\phi(X_i^T \hat{B})}{\Phi(X_i^T \hat{B})} 1_{\{1\}}(Y_i) -
	\frac{\phi(X_i^T \hat{B})}{1 - \Phi(X_i^T \hat{B})} 1_{\{0\}}(Y_i) \;.
	\]
	This, along with~\eqref{eq:fixedeqn2}, implies that
	\begin{equation} \nonumber
	Qv + \sum_{i=1}^{n} X_i \mathbb{E}(Z_i | B = \hat{B}, Y, X) = \Sigma
	\hat{B} \;.
	\end{equation}
	But this is equivalent to
	\[
	\int_{\mathbb{R}^p} \int_{\mathbb{R}^n} \beta'
	\pi_{B|Z,Y,X}(\beta'|z,Y,X) \pi_{Z|B,Y,X}(z|\hat{B},Y,X) \, \df z \,
	\df \beta' = \hat{B} \;,
	\]
	which is precisely~\eqref{eq:fixed point}.
\end{proof}

\begin{remark}
	\label{rem:a&c-pm}
	We should emphasize that~\eqref{eq:fixed point}, while interesting, is not essential to the proofs of our main results. It merely simplifies the process of establishing a drift condition.
\end{remark}

We will consider two different versions of \eqref{eq:generic}, both
centered at $\hat{B}$.  The first, which will be used in the large
$n$-small $p$ regime, is simply a centered version of $V_0$ given by
\[
V_1(\beta) = \big\|\Sigma^{1/2} (\beta - \hat{B}) \big\|^2 \;.
\]
In the large $p$-small $n$ regime, we assume that $Q$ is positive
definite (which is necessary for posterior propriety) and that $X$ is
full row rank, and we use the following drift function
\[
V_2(\beta) = \Big\| \big( X \Sigma^{-1} X^T \big)^{-1/2} X(\beta-\hat{B})
\Big\|^2 \;.
\]
In the next section, we establish two sets of d\&m conditions for the
A\&C chain based on $V_1$ and $V_2$.

\section{Results for the Albert and Chib Chain Part I: Fixed $n$ and $p$}
\label{sec:rp1}

\subsection{Drift inequalities for $V_1$ and $V_2$}
\label{ssec:v1&v2}

Define $g: \mathbb{R} \rightarrow \mathbb{R}$ as
\[
g(\theta) = \frac{\theta\phi(\theta)}{\Phi(\theta)} +
\left(\frac{\phi(\theta)}{\Phi(\theta)}\right)^2 \;.
\]
For any $\beta \in \mathbb{R}^p$, let $D(\beta)$ denote an $n \times
n$ diagonal matrix with $i$th diagonal element
\[
1 - g \big( X_i^T \beta \big) 1_{\{1\}}(Y_i) - g \big( \! -X_i^T \beta
\big) 1_{\{0\}}(Y_i) \;.
\]

\begin{lemma}
	\label{lem:drifts}
	If $V(\beta) = \norm{M (\beta - \hat{B})}^2, \, \beta \in \mathbb{R}^p$, where $M$ is any matrix
	with $p$ columns, then
	\begin{align*}
	\int_{\mathbb{R}^p} &V(\beta') k_{\mbox{\scriptsize{AC}}}(\beta,\beta') \, \df \beta'\\
	&\leq \sup_{t \in (0,1)} \big \| M \Sigma^{-1} X^T D(\hat{B} + t(\beta -
	\hat{B})) X (\beta - \hat{B}) \big \|^2 + 2 \, \mathrm{tr} \Big\{ M
	\Sigma^{-1} M^T \Big\} \;.
	\end{align*}
\end{lemma}

\begin{proof}
	Note that
	\begin{align*}
	\int_{\mathbb{R}^p} V(\beta') \, \pi_{B|Z,Y,X}&(\beta'|z,Y,X) \, \df
	\beta' \\
	&= \left\| M \Sigma^{-1} (X^Tz + Qv) - M\hat{B} \right\|^2 +
	\mathrm{tr} \left\{ M \Sigma^{-1} M^T \right\} \;.
	\end{align*}
	Moreover,
	\begin{align}
	\label{eq:need1}
	\int_{\mathbb{R}^n}  \|M \Sigma^{-1} &(X^Tz + Qv)  - M\hat{B}\|^2 \,
	\pi_{Z|B,Y,X}(z|\beta,Y,X) \,\df z \nonumber \\ & = \left\|M
	\Sigma^{-1} \left\{ X^T \mathbb{E}\left( Z|B=\beta,Y,X \right) + Qv
	\right\} - M \hat{B} \right\|^2 \nonumber \\ & \hspace{15mm} +
	\mathrm{tr} \left\{ M \Sigma^{-1} X^T \mathrm{var} \left(
	Z|B=\beta,Y,X \right) X \Sigma^{-1} M^T \right\} \;.
	\end{align}
	For two symmetric matrices of the same size, $M_1$ and $M_2$, we write
	$M_1 \leq M_2$ if $M_2 - M_1$ is non-negative definite.  By
	Lemma~\ref{lem:horrace} in Subsection~\ref{app:tn} of the Appendix,
	$\mathrm{var}(Z|B=\beta,Y,X) \leq I_n$.  It follows that
	\begin{equation}
	\label{eq:need2}
	M \Sigma^{-1}X^T \mathrm{var}\left( Z|B=\beta,Y,X \right) X
	\Sigma^{-1}M^T \leq M \Sigma^{-1} X^T X \Sigma^{-1} M^T \leq M
	\Sigma^{-1} M^T \;.
	\end{equation}
	Therefore,
	\begin{equation}\label{eq:drift-gen-2}
	\begin{aligned}
	\int_{\mathbb{R}^p} &V(\beta') k_{\mbox{\scriptsize{AC}}}(\beta,\beta') \, \df \beta' \\
	&=
	\int_{\mathbb{R}^p} V(\beta') \int_{\mathbb{R}^n}
	\pi_{B|Z,Y,X}(\beta'|z,Y,X) \pi_{Z|B,Y,X}(z|\beta,Y,X) \, \df z \,
	\df \beta'  \\ & \leq \left\| M \Sigma^{-1} \left\{ X^T
	\mathbb{E}\left( Z|B=\beta,Y,X \right) + Qv \right\} - M \hat{B}
	\right\|^2 + 2 \, \mathrm{tr} \left\{ M \Sigma^{-1} M^T \right\} \;.
	\end{aligned}
	\end{equation}
	Now, for $\alpha \in \mathbb{R}^p$, define
	\[
	\mu(\alpha) = M \Sigma^{-1} \left\{ X^T \mathbb{E}\left( Z|B=\hat{B} +
	\alpha,Y,X \right) + Qv \right\} - M \hat{B} \;.
	\]
	By Proposition~\ref{fixedmle}, we have
	\[
	\mu(0) = M \int_{\mathbb{R}^p} \beta k_{\mbox{\scriptsize{AC}}}(\hat{B},\beta) \, \df\beta
	- M\hat{B} = 0 \;.
	\]
	By the mean value theorem for vector-valued functions \citep[see,
	e.g.,][Theorem 5.19]{rudin1976principles}, for any $\alpha \in
	\mathbb{R}^p$,
	\begin{equation*} 
	\norm{\mu(\alpha)}^2 \leq \Bigg( \sup_{t \in(0,1)}\left\|
	\frac{\partial\mu(t\alpha)}{\partial t} \right\| \Bigg)^2 \;.
	\end{equation*} 
	Now, by results on truncated normal distributions in Subsection~\ref{app:tn} of the Appendix,
	\begin{align*}
	\frac{\partial \mu(t\alpha)}{\partial t} &= M \Sigma^{-1}
	\sum_{i=1}^{n} X_i \frac{\partial}{\partial t} \mathbb{E}(Z_i|B =
	\hat{B} + t\alpha, Y,X) \\ &= M \Sigma^{-1} \sum_{i=1}^{n} X_i
	\frac{\partial}{\partial t} \left\{ X_i^T(\hat{B} + t\alpha) \right\}
	\\ & \hspace{15mm} \times \frac{d}{d\theta} \left( \theta +
	1_{\{1\}}(Y_i) \frac{\phi(\theta)}{\Phi(\theta)} - 1_{\{0\}}(Y_i)
	\frac{\phi(\theta)}{1 - \Phi(\theta)} \right) \bigg|_{\theta =
		X_i^T(\hat{B} + t\alpha)} \\ &= M \Sigma^{-1} X^T \, D(\hat{B} + t\alpha)
	\, X \alpha \;.
	\end{align*}
	Hence,
	\begin{equation} \label{eq:meanvalue}
	\|\mu(\alpha)\|^2 \leq \sup_{t \in (0,1)} \big \| M
	\Sigma^{-1} X^T D(\hat{B} + t\alpha) X \alpha \big \|^2 \;.
	\end{equation}
	The result then follows from~\eqref{eq:drift-gen-2}
	and~\eqref{eq:meanvalue} by taking $\alpha = \beta-\hat{B}$.
\end{proof}

We now use Lemma~\ref{lem:drifts} to establish explicit drift
inequalities for $V_1$ and $V_2$.  We begin with $V_1$.  

\begin{proposition} \label{prop:V_1}
	For all $\beta \in \mathbb{R}^p$, we have
	\begin{align*}
	\int_{\mathbb{R}^p} V_1(\beta') &k_{\mbox{\scriptsize{AC}}}(\beta,\beta') \, \df
	\beta' \\
	&\leq \bigg( \sup_{t \in (0,1)} \sup_{\alpha \neq 0}
	\frac{\|\Sigma^{-1/2}X^TD(\hat{B} + t\alpha)X \alpha
		\|^2}{\|\Sigma^{1/2}\alpha\|^2} \bigg) V_1(\beta) + 2p \;.
	\end{align*}
\end{proposition}
\begin{proof}
	Taking $M = \Sigma^{1/2}$ in Lemma~\ref{lem:drifts} yields
	\begin{align*}
	\int_{\mathbb{R}^p} V_1(\beta') & k_{\mbox{\scriptsize{AC}}}(\beta,\beta') \, \df
	\beta' \\
	& \leq \sup_{t \in (0,1)} \big \| \Sigma^{-1/2} X^T
	D(\hat{B} + t(\beta - \hat{B})) X (\beta - \hat{B}) \big\|^2 + 2p \\ & \le
	\bigg( \sup_{t \in (0,1)} \sup_{\alpha \neq 0}
	\frac{\|\Sigma^{-1/2}X^TD(\hat{B} + t\alpha)X \alpha
		\|^2}{\|\Sigma^{1/2}\alpha\|^2} \bigg) V_1(\beta) + 2p \;.
	\end{align*}
\end{proof}
\noindent In Subsection~\ref{app:lambda<1} of the Appendix, we prove that 
\[
\sup_{t \in (0,1)} \sup_{\alpha \neq 0} \frac{\|\Sigma^{-1/2}X^TD(\hat{B}
	+ t\alpha)X \alpha \|^2}{\|\Sigma^{1/2}\alpha\|^2} < 1 \;.
\]

For a symmetric matrix $M$, let $\lambda_{\min}(M)$ and
$\lambda_{\max}(M)$ denote the smallest and largest eigenvalues of
$M$, respectively.  Here is the analogue of Proposition~\ref{prop:V_1}
for $V_2$.

\begin{proposition} 
	\label{prop:V_2}
	Assume that $X$ has full row rank.  Then for all $\beta \in
	\mathbb{R}^p$, we have
	\begin{equation*}
	\int_{\mathbb{R}^p} V_2(\beta') k_{\mbox{\scriptsize{AC}}}(\beta,\beta') \, \df \beta'
	\leq \left\{ \lambda^2_{\max} \left( X \Sigma^{-1} X^T \right) \right\}
	V_2(\beta) + 2n \;.
	\end{equation*}
\end{proposition}

\begin{proof}
	Taking $M = \big( X \Sigma^{-1} X^T \big)^{-1/2} X$ in
	Lemma~\ref{lem:drifts} and applying Lemma~\ref{lem:horrace} yields
	\begin{align*}
	&\int_{\mathbb{R}^p} V_2(\beta') k_{\mbox{\scriptsize{AC}}}(\beta,\beta') \, \df \beta'\\ &
	\leq \sup_{t \in (0,1)} \big \| \big( X \Sigma^{-1} X^T
	\big)^{1/2} D(\beta + t(\beta - \hat{B})) X (\beta - \hat{B}) \big\|^2 + 2n \\ 
	& \le  \sup_{t \in (0,1)} \lambda_{\max}^2 \big\{\big( X \Sigma^{-1} X^T
	\big)^{1/2} D(\hat{B} + t(\beta - \hat{B})) \big( X \Sigma^{-1} X^T
	\big)^{1/2} \big\} V_2(\beta) + 2n \\
	& \le
	\left\{ \lambda^2_{\max} \left( X \Sigma^{-1} X^T \right) \right\}
	V_2(\beta) + 2n \;.
	\end{align*}
\end{proof}

\subsection{Drift and minorization for the Albert and Chib chain based on $V_1$}
\label{ssec:V_1}

In this subsection, we exploit the flipped chain idea described in
Section~\ref{sec:d_m}.  In particular, $V_1$ is used to establish
d\&m conditions for a flipped chain that has the
same geometric convergence rate as A\&C's chain.
Later, in Section~\ref{sec:rp2}, we will use these results to prove
asymptotic stability as $n \to \infty$.

Note that $\pi_{B|Z,Y,X}(\beta|Z,Y,X)$ depends on the $n$-dimensional
vector~$Z$ only through $X^TZ$, which is a one-to-one function of
the following $p$-dimensional vector:
\[
\Gamma := \Sigma^{1/2} \left\{ \Sigma^{-1} \left(X^TZ + Qv\right) -
\hat{B} \right\} \;.
\] 
Hence, we can represent the Mtd of the A\&C chain as follows:
\[
k_{\mbox{\scriptsize{AC}}}(\beta,\beta') = \int_{\mathbb{R}^p}
\pi_{B|\Gamma,Y,X}(\beta'|\gamma,Y,X)
\pi_{\Gamma|B,Y,X}(\gamma|\beta,Y,X) \, \df \gamma \;.
\]
Recalling the discussion in Section~\ref{sec:d_m}, this maneuver seems
to represent progress since we have replaced $n$ with $p$.  However,
it is difficult to establish a minorization condition using this
version of $k_{\mbox{\scriptsize{AC}}}(\beta,\beta')$ because
$\pi_{\Gamma|B,Y,X}(\gamma|\beta,Y,X)$ lacks a closed form.  On the
other hand, this chain has the same geometric convergence rate as the
flipped chain defined by the following Mtd:
\begin{align*}
\tilde{k}_{\mbox{\scriptsize{AC}}}(\gamma,\gamma') &:= \tilde{k}_{\mbox{\scriptsize{AC}}}(\gamma,\gamma'; Y, X)\\
&= \int_{\mathbb{R}^p} \pi_{\Gamma|B,Y,X} (\gamma'|\beta,Y,X)
\pi_{B|\Gamma,Y,X}(\beta|\gamma,Y,X) \, \df \beta \;.
\end{align*}
Constructing a minorization condition for this Mtd is much less
daunting since
\[
B|\Gamma,Y,X \sim N \big( \Sigma^{-1/2}\Gamma + \hat{B}, \Sigma^{-1} \big)
\;.
\]
Here is the main result of this subsection.

\begin{proposition}
	\label{prop:driftmin-n}
	Let $\tilde{V}_1(\gamma) = \norm{\gamma}^2$.  The Mtd
	$\tilde{k}_{\mbox{\scriptsize{AC}}}$ satisfies the drift condition
	\[
	\int_{\mathbb{R}^p} \tilde{V}_1(\gamma')
	\tilde{k}_{\mbox{\scriptsize{AC}}}(\gamma,\gamma') \, \df \gamma' \leq \lambda
	\tilde{V}_1(\gamma) + L \;,
	\]
	where $L = p(1+\lambda)$, and
	\[
	\lambda = \sup_{t \in (0,1)} \sup_{\alpha \neq 0} \frac{\|\Sigma^{-1/2}X^TD(\hat{B} + t\alpha)X \alpha \|^2}{\|\Sigma^{1/2}\alpha\|^2}  \;.
	\]
	Moreover, for $d > 2L/(1-\lambda)$, $\tilde{k}_{\mbox{\scriptsize{AC}}}$ satisfies
	\[
	\tilde{k}_{\mbox{\scriptsize{AC}}}(\gamma,\gamma') \geq \varepsilon q(\gamma') \;,
	\]
	where $q: \mathbb{R}^p \rightarrow [0,\infty)$ is a pdf, and
	$\varepsilon = 2^{-p/2} e^{-d}$.
\end{proposition}

\begin{proof}
	We begin with the drift.  It's easy to verify that
	\[
	\tilde{V}_1(\gamma) = \int_{\mathbb{R}^p} V_1(\beta) \pi_{B|\Gamma,Y,X}(\beta|\gamma,Y,X) \,
	\df \beta - p \;.
	\]
	We now use the techniques described at the end of
	Section~\ref{sec:d_m} to convert the drift inequality in
	Proposition~\ref{prop:V_1} into a drift inequality for the flipped
	chain.  We know that
	\[
	\int_{\mathbb{R}^p} V_1(\beta') k_{\mbox{\scriptsize{AC}}}(\beta,\beta') \, \df \beta' \leq \lambda V_1(\beta) + 2p \;.
	\]
	Then taking $c = -p$ in Proposition~\ref{prop:bad_ref} yields
	\[
	\int_{\mathbb{R}^p} \tilde{V}_1(\gamma')
	\tilde{k}_{\mbox{\scriptsize{AC}}}(\gamma,\gamma') \, \df \gamma' \leq \lambda
	\tilde{V}_1(\gamma) + p(1+\lambda) \;.
	\]
	As explained in Section~\ref{sec:d_m}, to establish the minorization
	condition, it suffices to show that there exists a pdf $\nu(\beta) :=
	\nu(\beta|Y,X)$ such that
	\begin{equation}
	\label{eq:n-min-flipped-pre}
	\pi_{B|\Gamma,Y,X}(\beta|\gamma,Y,X) \geq \varepsilon \nu(\beta)
	\end{equation}
	whenever $\tilde{V}_1(\gamma) \leq d$.  Recall that
	\[
	B|\Gamma,Y,X \sim N \big( \Sigma^{-1/2}\Gamma + \hat{B}, \Sigma^{-1}
	\big) \;.
	\]
	Define 
	\begin{align*}
	\nu_1(\beta) & = \inf_{\gamma:\tilde{V}_1(\gamma) \leq d}
	\pi_{B|\Gamma,Y,X}(\beta|\gamma,Y,X) \\ & =
	\inf_{\gamma:\tilde{V}_1(\gamma) \leq d}
	\frac{|\Sigma|^{1/2}}{(2\pi)^{p/2}} \exp\left\{ -\frac{1}{2} \left\|
	\Sigma^{1/2} \left( \beta - \hat{B} - \Sigma^{-1/2}\gamma \right)
	\right\|^2 \right\} \;.
	\end{align*}
	Then $\nu(\beta) = \nu_1(\beta)/\int_{\mathbb{R}^p} \nu_1(\beta') \,
	\df \beta'$ is a pdf, and whenever $\tilde{V}_1(\gamma) \le d$,
	\[
	\pi_{B|\Gamma,Y,X}(\beta|\gamma,Y,X) \geq \left( \int_{\mathbb{R}^p}
	\nu_1(\beta') \, \df\beta' \right) \nu(\beta) \;.
	\]
	This is~\eqref{eq:n-min-flipped-pre} with
	\[
	\varepsilon = \int_{\mathbb{R}^p} \nu_1(\beta) \, \df \beta =
	(2\pi)^{-p/2} \int_{\mathbb{R}^p} \inf_{\gamma: \|\gamma\|^2 \leq d}
	\exp\left( -\frac{1}{2} \left\|\beta-\gamma\right\|^2 \right) \, \df
	\beta \;.
	\]
	Finally, since $\|\beta-\gamma\|^2 \leq 2(\|\beta\|^2 +
	\|\gamma\|^2)$,
	\[
	\varepsilon \geq (2\pi)^{-p/2} \int_{\mathbb{R}^p} \inf_{\gamma:
		\|\gamma\|^2 \leq d} \exp\left( -\left\|\beta\right\|^2 -
	\|\gamma\|^2 \right) \, \df \beta = 2^{-p/2} e^{-d} \;.
	\]
\end{proof}

Mainly, Proposition~\ref{prop:driftmin-n} will be used to establish
asymptotic stability results for A\&C's chain in the large~$n$,
small~$p$ regime.  Indeed, since the Markov chains defined by $k_{\mbox{\scriptsize{AC}}}$
and $\tilde{k}_{\mbox{\scriptsize{AC}}}$ have the same geometric convergence rate,
$\hat{\rho}$ calculated using \eqref{eq:ros_bound}
with~$\lambda$,~$L$, and~$\varepsilon$ from
Proposition~\ref{prop:driftmin-n} is an upper bound on $\rho_* =
\rho_*(X,Y)$ for the A\&C chain.  On the other hand,
Proposition~\ref{prop:driftmin-n} can also be used in conjunction with
Theorem~\ref{thm:rosen} and Corollary~\ref{cor:tv_flip} to get
computable bounds on the total variation distance to stationarity for
the A\&C chain for fixed $n$ and $p$.  In order to state the result,
we require a bit of notation.  For an integer $m \geq 1$, let $k^{(m)}_{\mbox{\scriptsize{AC}}}: \mathbb{R}^p \times \mathbb{R}^p \to \mathbb{R}_+$ be the chain's $m$-step Mtd. For $\beta \in \mathbb{R}^p$, let $\varphi(\beta) =  \mathbb{E}\left(Z|B=\beta,Y,X \right) \in \mathbb{R}^n$. Then results in
Subsection~\ref{app:tn} of the Appendix show that the $i$th element of
$\varphi(\beta)$ is given by
\[
X_i^T\beta + \frac{\phi(X_i^T\beta)}{\Phi(X_i^T\beta)} 1_{\{1\}}(Y_i)
- \frac{\phi(X_i^T\beta)}{1 - \Phi(X_i^T\beta)} 1_{\{0\}}(Y_i) \;.
\]

\begin{proposition}
	\label{prop:tvv1}
	If $\hat{\rho}$ is calculated using \eqref{eq:ros_bound}
	with~$\lambda$,~$L$, and~$\varepsilon$ from
	Proposition~\ref{prop:driftmin-n}, then for $m \ge 1$ and
	$\beta \in \mathbb{R}^p$,
	\[
	\int_{\mathbb{R}^p}
	\big|k^{(m)}_{\mbox{\scriptsize{AC}}}(\beta,\beta';Y,X) -
	\pi_{B|Y,X}(\beta'|Y,X) \big| \, \df \beta' \le H(\beta) \,
	\hat{\rho}^{m-1} \;,
	\]
	where
	\begin{equation*}
	H(\beta) = 2 + \frac{L}{1-\lambda} + \mathrm{tr}\left( X
	\Sigma^{-1}X^T \right) + \left\| \Sigma^{1/2} \left\{ \Sigma^{-1}
	\left(X^T\varphi(\beta) + Qv\right) - \hat{B} \right\} \right\|^2 \;.
	\end{equation*}
\end{proposition}

\begin{proof}
	We simply apply Corollary~\ref{cor:tv_flip}.  Putting $M =
	\Sigma^{1/2}$ in \eqref{eq:need1}, we have
	\begin{align*}
	\int_{\mathbb{R}^p} \tilde{V}_1(\gamma) &\, \pi_{\Gamma|B,Y,X}
	(\gamma|\beta,Y,X) \, \df \gamma \\ 
	& = \int_{\mathbb{R}^n} \Big \|
	\Sigma^{1/2} \left\{ \Sigma^{-1} \left(X^Tz + Qv\right) - \hat{B} \right\}
	\Big \|^2 \, \pi_{Z|B,Y,X}(z|\beta,Y,X) \, \df z \\ & = \Big \|
	\Sigma^{1/2} \left\{ \Sigma^{-1} \left(X^T \varphi(\beta) + Qv\right)
	- \hat{B} \right\} \Big \|^2 \\ & \hspace{20mm} + \mathrm{tr} \left\{
	\Sigma^{-1/2} X^T \mathrm{var}(Z|B=\beta,Y,X) X \Sigma^{-1/2} \right\}
	\;.
	\end{align*}
	A calculation similar to \eqref{eq:need2} shows that
	\[
	\mathrm{tr} \left\{ \Sigma^{-1/2} X^T \mathrm{var}(Z|B=\beta,Y,X) X
	\Sigma^{-1/2} \right\} \leq \mathrm{tr}\left\{ X \Sigma^{-1}X^T
	\right\} \;,
	\]
	and the result follows.
\end{proof}

Calculating~$\lambda$ in Proposition~\ref{prop:driftmin-n} calls for
maximization of a function on $(0,1) \times \mathbb{R}^p$, which may
be difficult.  Here we provide an upper bound on~$\lambda$ that's easy
to compute when~$p$ is small. Let $\{S_j\}_{j=1}^{2^p}$ denote the
open orthants of $\mathbb{R}^p$. For instance, if $p = 2$, then
$S_1,S_2,S_3$ and $S_4$ are the open quadrants of the real plane.
Define $W(S_j) = W(S_j;X,Y)$ as follows
\[
W(S_j) = \sum_{X_i \in S_j} X_i 1_{\{0\}}(Y_i) X_i^T +
\sum_{X_i\in-S_j} X_i 1_{\{1\}}(Y_i) X_i^T \;.
\]
The following result is proven in Subsection~\ref{app:lambdasimple} of the Appendix.
\begin{proposition} 
	\label{prop:lambdasimple}
	An upper bound on $\lambda^{1/2}$ in Proposition~\ref{prop:driftmin-n} is
	\[
	\lambda_{\max} \left( \Sigma^{-1/2} X^TX
	\Sigma^{-1/2} \right) - \frac{2}{\pi} \min_{1 \leq j \leq 2^p}
	\lambda_{\min} \left( \Sigma^{-1/2} W(S_j) \Sigma^{-1/2} \right)  \;.
	\]
\end{proposition}
\noindent If this upper bound is strictly less than~$1$ (which is
always true when~$Q$ is positive definite), then one can
replace~$\lambda$ with the square of this bound in
Proposition~\ref{prop:driftmin-n}.

\subsection{Drift and minorization for the Albert and Chib chain based on $V_2$}
\label{ssec:V_2}

The A\&C chain has the same convergence rate as the flipped chain
defined by the following Mtd:
\begin{align*}
\check{k}_{\mbox{\scriptsize{AC}}}(z,z') &:= \check{k}_{\mbox{\scriptsize{AC}}}(z,z';Y,X) \\ &= \int_{\mathbb{R}^p}
\pi_{Z|B,Y,X}(z'|\beta,Y,X) \pi_{B|Z,Y,X}(\beta|z,Y,X) \, \df \beta
\;.
\end{align*}
In this subsection, we use $V_2$ to establish d\&m
conditions for this chain, and these will be used later to prove
asymptotic stability as $p \to \infty$.  First, for $z \in
\mathbb{R}^n$, define
\[
w(z) = \big(X \Sigma^{-1} X^T \big)^{-1/2} X \big \{ \Sigma^{-1}
(X^Tz+Qv) - \hat{B} \big \} \;.
\]
Now define $\check{V}_2: \mathbb{R}^n \rightarrow [0,\infty)$ as
$\check{V}_2(z) = \norm{w(z)}^2$.

\begin{proposition}
	\label{prop:driftmin-p}
	The Mtd $\check{k}_{\mbox{\scriptsize{AC}}}$ satisfies the drift condition
	\begin{equation}
	\label{eq:pf1}
	\int_{\mathbb{R}^n} \check{V}_2(z') \check{k}_{\mbox{\scriptsize{AC}}}(z,z') \df z' \leq
	\lambda \check{V}_2(z) + L \;,
	\end{equation}
	where $\lambda = \lambda^2_{\max} \big \{ X \Sigma^{-1} X^T \big \}$
	and $L = n(1+\lambda)$.  Moreover, for $d > 2L/(1-\lambda)$,
	$\check{k}_{\mbox{\scriptsize{AC}}}$ satisfies
	\[
	\check{k}_{\mbox{\scriptsize{AC}}}(z,z') \geq \varepsilon q(z) \;,
	\]
	where $q: \mathbb{R}^n \rightarrow [0,\infty)$ is a pdf, and
	$\varepsilon = 2^{-n/2} e^{-d}$.
\end{proposition}

\begin{proof}
	It's easy to verify that
	\[
	\check{V}_2(z) = \int_{\mathbb{R}^p} V_2(\beta') \pi_{B|Z,Y,X}(\beta'|z,Y,X) \, \df
	\beta' - n \;.
	\]
	We know from Proposition~\ref{prop:V_2} that
	\begin{align*}
	\int_{\mathbb{R}^p} V_2(\beta') \bigg\{ \int_{\mathbb{R}^n}
	\pi_{B|Z,Y,X}(\beta'|z',Y,X) \, &\pi_{Z|B,Y,X}(z'|\beta,Y,X) \, \df
	z' \bigg\} \, \df \beta' \\ 
	&\le \lambda V_2(\beta) + 2n \,.
	\end{align*}
	As in Proposition~\ref{prop:bad_ref}, multiplying both sides of the above inequality by the conditional density
	$\pi_{B|Z,Y,X}(\beta|z,Y,X)$ and integrating with respect to $\beta$
	yields \eqref{eq:pf1}.
	
	We now move on the the minorization condition.  Note that
	$\pi_{Z|B,Y,X}(z|B,Y,X)$ depends on $B$ only through $XB$, which is a
	one-to-one function of the $n$-dimensional vector
	\[
	A := \big( X \Sigma^{-1} X^T \big)^{-1/2} X(B-\hat{B}) \;.
	\] 
	Hence, $\check{k}_{\mbox{\scriptsize{AC}}}(z,z')$ can be re-expressed as
	\[
	\check{k}_{\mbox{\scriptsize{AC}}}(z,z') = \int_{\mathbb{R}^n}
	\pi_{Z|A,Y,X}(z'|\alpha,Y,X) \pi_{A|Z,Y,X}(\alpha|z,Y,X) \, \df \alpha
	\;,
	\]
	where $A|Z,Y,X \sim N( w(Z), I_n)$.  To get the minorization
	condition, we will construct a pdf $\nu(\alpha) := \nu(\alpha|Y,X)$
	such that
	\begin{equation*}
	\pi_{A|Z,Y,X}(\alpha|z,Y,X) \geq \varepsilon \nu(\alpha)
	\end{equation*}
	whenever $\check{V}_2(z) \leq d$.  Define 
	\begin{equation*}
	\nu_1(\alpha) = \inf_{z:\check{V}_2(z) \leq d}
	\pi_{A|Z,Y,X}(\alpha|z,Y,X) = \inf_{w:w^2 \leq d} (2\pi)^{-n/2} \exp
	\Big( -\frac{1}{2} \|\alpha - w \|^2 \Big) \;.
	\end{equation*}
	Then $\nu(\alpha) = \nu_1(\alpha)/\int_{\mathbb{R}^n} \nu_1(\alpha')
	\, \df \alpha'$ is a pdf, and
	\begin{align*}
	\varepsilon = (2\pi)^{-n/2} \int_{\mathbb{R}^n} \inf_{w:w^2 \leq d}
	&\exp \Big( -\frac{1}{2} \|\alpha - w \|^2 \Big) \, \df \alpha \\
	&\ge
	(2\pi)^{-n/2} e^{-d} \int_{\mathbb{R}^n} \exp( -\|\alpha \|^2) \, \df
	\alpha = 2^{-n/2} e^{-d} \;.
	\end{align*}
\end{proof}

The next result is the analogue of Proposition~\ref{prop:tvv1}.  The
proof is omitted as it is essentially the same as the proof of
the said proposition.

\begin{proposition}
	\label{prop:tvv2}
	Assume that $X$ has full row rank.  Let $\lambda$, $L$, and
	$\varepsilon$ be as in Proposition~\ref{prop:driftmin-p}.  If
	$\hat{\rho}$ is calculated using \eqref{eq:ros_bound}, then for $m \ge
	1$ and $\beta \in \mathbb{R}^p$,
	\[
	\int_{\mathbb{R}^p}
	\big|k^{(m)}_{\mbox{\scriptsize{AC}}}(\beta,\beta';Y,X) -
	\pi_{B|Y,X}(\beta'|Y,X) \big| \, \df \beta' \le H(\beta) \,
	\hat{\rho}^{m-1} \;,
	\]
	where
	\begin{align*}
	H(\beta) = 2 + \frac{L}{1-\lambda} + &\mathrm{tr}\left( X
	\Sigma^{-1}X^T \right) \\
	& + \left\| \big( X \Sigma^{-1} X^T \big
	)^{-1/2} X \left\{ \Sigma^{-1} \left(X^T\varphi(\beta) + Qv\right)
	- \hat{B} \right\} \right\|^2 \;.
	\end{align*}
\end{proposition}

Let $\rho_* = \rho_*(X,Y)$ denote the geometric convergence rate of
the A\&C chain.  In the next section, which is the heart of the paper,
we develop general convergence complexity results showing that, under
weak regularity conditions, $\rho_*$ is bounded away from 1 both as $n
\rightarrow \infty$ (for fixed $p$), and as $p \rightarrow \infty$
(for fixed $n$).

\section{Results for the Albert and Chib Chain Part II: Asymptotics}
\label{sec:rp2}

\subsection{Large $n$, small $p$}
\label{sec:lnsp}

In this section, we consider the case where $p$ is fixed and $n$
grows.  In particular, we are interested in what happens to the
geometric convergence rate of the A\&C chain in this setting.  Recall
that the prior on $B$ is
\[
\omega(\beta) \propto \exp\{ -(\beta-v)^TQ(\beta-v)/2 \} \;.
\]
Since $p$ is fixed, so is the prior.  Hence, the hyperparameters $v$
and $Q$ will remain fixed throughout this subsection.  In
Subsection~\ref{ssec:basics}, we introduced the data set $\mathcal{D}
:= \{(X_i,Y_i)\}_{i=1}^n$.  We now let~$n$ vary, and consider a sequence of data sets,
$\mathcal{D}_n := \{(X_i,Y_i)\}_{i=1}^n, \, n \geq 1$.  So, each time $n$ increases
by 1, we are are given a new $p \times 1$ covariate vector and
a corresponding binary response.  In order to study the
asymptotics, we assume that the $(X_i,Y_i)$ pairs are generated
according to a random mechanism that is consistent with the probit
regression model.  In particular, we make the following assumptions:
\begin{enumerate}
	\item[($A1$)] The pairs $\{(X_i,Y_i)\}_{i=1}^\infty$ are iid random
	vectors such that \[Y_i|X_i \sim \mbox{Bernoulli}(G(X_i)),\] where
	$G : \mathbb{R}^p \rightarrow (0,1)$ is a measurable function;
	\item[($A2$)] $\mathbb{E}X_1X_1^T$ is finite and positive definite;
	\item[($A3$)] For $j \in \{1,2,\dots,2^p\}$ and $\beta \ne 0$,
	\[
	\mathbb{P} \Big( \big(1_{S_j}(X_1) + 1_{S_j}(-X_1) \big) X_1^T\beta
	\neq 0 \Big) > 0 \;.
	\]
\end{enumerate}

Assumption (A1) contains the probit model as a special case.  Thus,
our results concerning the asymptotic behavior of A\&C's Markov chain
do not require strict adherence of the data to the probit regression
model.  The main reason for assuming (A2) is to guarantee that, almost
surely, $X$ will eventually have full column rank.  This is a
necessary condition for posterior propriety when $Q=0$.  While (A3) is
rather technical, it is clearly satisfied if $X_1$ follows a
distribution admitting a pdf (with respect to Lebesgue measure) that
is positive over some open ball $\{\gamma \in \mathbb{R}^p:
\|\gamma\|^2 \le c \}$, where $c > 0$.  It is shown in
Subsection~\ref{app:intercept} of the Appendix that (A3) also allows for an intercept.
That is, even if the first component of $X_1$ is~1 (constant), then as
long as the remaining $p-1$ components satisfy the density condition
described above, (A3) is still satisfied.

It was assumed throughout Section~\ref{sec:bayesprobitac} that the
posterior distribution is proper.  Of course, for a fixed data set,
this is check-able.  All we need in the large $n$, small $p$ regime is
a guarantee that the posterior is proper for all large $n$, almost
surely.  The following result is proven in
Subsection~\ref{app:proper-n}.

\begin{proposition}
	\label{prop:proper-n}
	Under assumptions (A1)-(A3), almost surely, the posterior
	distribution is proper for all sufficiently large~$n$.
\end{proposition}

For fixed $n$, $\mathcal{D}_n = \{(X_i,Y_i)\}_{i=1}^n$ represents the
first $n$ (random) covariate vectors and responses.  Let
$\rho_*(\mathcal{D}_n)$ denote the (random) geometric convergence rate
of the corresponding A\&C chain.  Here is one of our main results.

\begin{theorem} 
	\label{thm:stable-n}
	If (A1)-(A3) hold, then there exists a constant $\rho < 1$ such
	that, almost surely,
	\[
	\limsup_{n \to \infty} \rho_*(\mathcal{D}_n) \leq \rho \;.
	\]
\end{theorem}

\begin{proof}
	Let $\hat{\rho}(\mathcal{D}_n)$ denote the upper bound on
	$\rho_*(\mathcal{D}_n)$ that is based on $\lambda$, $L$, and
	$\varepsilon$ from Proposition~\ref{prop:driftmin-n}.  We prove the
	result by showing that, almost surely, $\limsup_{n \to \infty}
	\hat{\rho}(\mathcal{D}_n) \le \rho < 1$.  Note that $L=p(1+\lambda)$
	and $\varepsilon = 2^{-p/2} e^{-d}$ (where $d > 2L/(1-\lambda)$).
	Thus, control over $\lambda$ provides control over $L$ and
	$\varepsilon$ as well.  In particular, to prove the result it
	suffices to show that there exists a constant $c \in [0,1)$, such
	that, almost surely,
	\begin{equation}
	\label{eq:lsl}
	\limsup_{n \to \infty} \lambda(\mathcal{D}_n) \le c \;.
	\end{equation}
	Noting that $\Sigma^{-1/2} X^TX \Sigma^{-1/2} \leq I_p$, we have, by
	Proposition~\ref{prop:lambdasimple},
	\begin{equation}
	\label{eq:driftproof2}
	\begin{aligned}
	\lambda^{1/2} \leq 1 - \frac{2}{\pi} \min_{1 \leq j \leq 2^p}
	\lambda_{\min} \Bigg\{ &\left(\frac{\Sigma}{n}\right)^{-1/2} \Bigg( \frac{1}{n}
	\sum_{X_i \in S_j} X_i 1_{\{0\}}(Y_i) X_i^T \\& + \frac{1}{n}
	\sum_{X_i \in -S_j} X_i 1_{\{1\}}(Y_i) X_i^T \Bigg)
	\left(\frac{\Sigma}{n}\right)^{-1/2} \Bigg \} \;.
	\end{aligned}
	\end{equation}
	Fix $j \in \{1,2,\dots,2^p\}$.  By (A1) and the strong law, almost
	surely,
	\begin{align*}
	\lim_{n \to \infty}  \left(\frac{\Sigma}{n}\right)^{-1/2} \Bigg( \frac{1}{n} \sum_{X_i
		\in S_j} X_i &1_{\{0\}}(Y_i) X_i^T + \frac{1}{n} \sum_{X_i \in
		-S_j} X_i 1_{\{1\}}(Y_i) X_i^T \Bigg) \left(\frac{\Sigma}{n}\right)^{-1/2} \\  =
	\big( \mathbb{E}X_1X_1^T \big)^{-1/2} &\Big\{ \mathbb{E}X_1X_1^T
	1_{S_j}(X_1) (1 - G(X_1)) \\
	& + \mathbb{E}X_1X_1^T 1_{S_j}(-X_1) G(X_1)
	\Big\} \big( \mathbb{E}X_1X_1^T \big)^{-1/2} \;.
	\end{align*}
	It follows from \eqref{eq:m2r} and Lemma~\ref{lem:matrixlimit} that,
	almost surely,
	\begin{equation}\label{eq:driftproof3}
	\begin{aligned}
	\lim_{n \to \infty} \lambda_{\min} \Bigg\{  \left(\frac{\Sigma}{n}\right)^{-1/2} & \Bigg(
	\frac{1}{n} \sum_{X_i \in S_j} X_i 1_{\{0\}}(Y_i) X_i^T \\
	& + \frac{1}{n} \sum_{X_i \in -S_j} X_i 1_{\{1\}}(Y_i) X_i^T \Bigg)
	\left(\frac{\Sigma}{n}\right)^{-1/2} \Bigg \} \\  \ge \lambda^{-1}_{\max}
	\big( \mathbb{E}X_1X_1^T \big) & \lambda_{\min} \Big\{ \mathbb{E}X_1X_1^T
	1_{S_j}(X_1) (1 - G(X_1)) \\
	& + \mathbb{E}X_1X_1^T 1_{S_j}(-X_1) G(X_1)
	\Big\} \;.
	\end{aligned}
	\end{equation}
	By (A2), $\lambda^{-1}_{\max} \big( \mathbb{E}X_1X_1^T \big)>0$.
	Hence by \eqref{eq:driftproof2} and \eqref{eq:driftproof3}, to show
	that \eqref{eq:lsl} holds, almost surely, it is enough to show that
	\[
	\lambda_{\min} \Big\{ \mathbb{E}X_1X_1^T 1_{S_j}(X_1) (1 - G(X_1)) +
	\mathbb{E}X_1X_1^T 1_{S_j}(-X_1) G(X_1) \Big\} > 0 \;.
	\]
	By (A3), for any $\beta \neq 0$,
	\begin{equation}
	\label{eq:A3or}
	\mathbb{P} \Big( 1_{S_j}(X_1) X_1^T\beta \neq 0 \;\; \text{or} \;\;
	1_{S_j}(-X_1) X_1^T \beta \neq 0 \Big) > 0 \;.
	\end{equation}
	Since $0 < G(X_1) < 1$, 
	\eqref{eq:A3or} implies that,
	for any $\beta \neq 0$, 
	\[
	\mathbb{P} \Big( \beta^T \big\{ X_1X_1^T 1_{S_j}(X_1) \big(1 - G(X_1)
	\big) \big\} \beta + \beta^T \big\{ X_1X_1^T 1_{S_j}(-X_1) G(X_1)
	\big\} \beta > 0 \Big) > 0 \;
	\]
	As a result, for any $\beta \in \mathbb{R}^p$ such that $\|\beta\|^2 =
	1$,
	\begin{equation}
	\label{eq:functionbeta}
	\beta^T \Big\{ \mathbb{E}X_1X_1^T 1_{S_j}(X_1) (1 - G(X_1)) +
	\mathbb{E}X_1X_1^T 1_{S_j}(-X_1) G(X_1) \Big\} \beta > 0 \;.
	\end{equation}
	It follows from (A2) that the left-hand side
	of~\eqref{eq:functionbeta} is continuous in~$\beta$.  Hence, we can
	take infimum on both sides of the inequality with respect to~$\beta$
	and retain the greater-than symbol, which yields
	\[
	\lambda_{\min} \Big\{ \mathbb{E}X_1X_1^T 1_{S_j}(X_1) (1 - G(X_1)) +
	\mathbb{E}X_1X_1^T 1_{S_j}(-X_1) G(X_1) \Big\} > 0 \;,
	\]
	and the result follows.
\end{proof}

\begin{remark}
	From the proof of Theorem~\ref{thm:stable-n} it's easy to see that the asymptotic bound~$\rho$ in the said theorem is unaffected by the precision matrix~$Q$. This is because the effect of the prior is overshadowed by the increasing amount of data as $n \to \infty$.
\end{remark}

Theorem~\ref{thm:stable-n} shows that, under weak regularity
conditions on the random mechanism that generates $\mathcal{D}_n =
\{(X_i,Y_i)\}_{i=1}^n$, A\&C's MCMC algorithm scales well with $n$.
\citet{johndrow2016inefficiency} studied the convergence rate of the
A\&C chain as $n \rightarrow \infty$ for a particular \textit{fixed}
sequence of covariates and responses.  Suppose that $p=1$, $Q>0$ is a
constant, $v = 0$, and $X_1 = X_2 = \cdots = X_n = 1$.
They showed that if \textit{all} the
Bernoulli trials result in success, i.e., $Y_1 = Y_2 = \cdots = Y_n =
1$, then $\lim_{n \to \infty} \rho_{**}(\mathcal{D}_n) = 1$.  That is, in terms of $L^2$-geometric convergence rate, the convergence
is arbitrarily slow for sufficiently large~$n$.  As we now explain,
our results can be used to show that, in
\pcite{johndrow2016inefficiency} setting, almost any other sequence of
responses leads to well behaved convergence rates.  Let
$\{Y_i\}_{i=1}^\infty$ denote a fixed sequence of binary responses,
and define $\hat{p}_n = n^{-1} \sum_{i=1}^n Y_i$.  It follows from
Propositions~\ref{prop:driftmin-n} and~\ref{prop:lambdasimple} that
the A\&C chain satisfies d\&m conditions with
\[
\lambda = \bigg[ \frac{n}{n+Q} - \frac{2}{\pi} \frac{n}{n+Q} \big\{
\hat{p}_n \wedge ( 1- \hat{p}_n ) \big\} \bigg]^2 \; \leq \left[ 1
- \frac{2}{\pi}\{\hat{p}_n \wedge (1-\hat{p}_n)\} \right]^2,
\]
$L = 1+\lambda$, and $\varepsilon = 2^{-1/2} e^{-d}$ for $d >
2L/(1-\lambda)$.  For any fixed~$n$, suppose that there exist $c_1,
c_2 \in (0,1)$ such that $c_1 \leq \hat{p}_n \leq c_2$, then
(using~\eqref{eq:ros_bound}) one can find $\rho<1$, which depends
\textit{only} on $c_1 \wedge ( 1- c_2 )$, such that $
\rho_*(\mathcal{D}_n) \le \rho \;.  $ It now follows that the
geometric convergence rates, $\rho_*(\mathcal{D}_n)$, are eventually
bounded away from 1 so long as $0 < \liminf_{n \to \infty} \hat{p}_n
\le \limsup_{n \to \infty} \hat{p}_n < 1$. (It is important to note
that $\mathcal{D}_n$ is \textit{not} random here.)  Moreover, an
analogous result holds for $\rho_{**}(\mathcal{D}_n) $. Here's a
formal statement.

\begin{corollary}
	\label{corr:johndrow}
	For the intercept-only model described above, if \[0 < \liminf_{n \to \infty} \hat{p}_n \le \limsup_{n \to \infty}
	\hat{p}_n < 1,\] then for any
	\[
	\delta \in \left(0, 1 - \left[ 1 - \frac{2}{\pi} \left\{\liminf_{n
		\to \infty}\hat{p}_n \wedge \left(1-\limsup_{n \to
		\infty}\hat{p}_n \right) \right\} \right]^2 \right),
	\]
	$\limsup_{n \to \infty} \rho_*(\mathcal{D}_n) \le\rho < 1$, and
	$\limsup_{n \to \infty} \rho_{**}(\mathcal{D}_n) \le \rho < 1$,
	where~$\rho$ equals~$\hat{\rho}$ in~\eqref{eq:ros_bound} with
	\[
	\lambda = \left[ 1 - \frac{2}{\pi} \left\{\liminf_{n \to
		\infty}\hat{p}_n \wedge \left(1-\limsup_{n \to \infty}\hat{p}_n
	\right) \right\} \right]^2 + \delta \;,
	\]
	$L = 1+\lambda \,$, and $\varepsilon = 2^{-1/2} e^{-d}$, where $d >
	2L/(1-\lambda)$.
\end{corollary}
\begin{proof}
	It suffices to prove the result for $\rho_{**}$, since the
	argument for $\rho_{*}$ has already been provided.  Fix an arbitrary~$\delta$. By Proposition~\ref{prop:tvv1}, when~$n$ is
	sufficiently large, for any $\beta \in \mathbb{R}$ and $m \geq
	1$,
	\[
	\int_{\mathbb{R}}
	\big|k^{(m)}_{\mbox{\scriptsize{AC}}}(\beta,\beta';Y,X) -
	\pi_{B|Y,X}(\beta'|Y,X) \big| \, \df \beta' \le H(\beta) \,
	\rho^{m-1} \;,
	\]
	where $H(\beta)$ is given in the said proposition.  Let $\Pi$
	be the probability measure corresponding to the posterior
	density, $\pi_{B|Y,X}(\beta|Y,X)$.  Then for any probability
	measure $\nu \in L^2(\Pi)$ and $m \geq 1$,
	\begin{align*}
	\int_{\mathbb{R}} \bigg|\int_{\mathbb{R}} 
	k^{(m)}_{\mbox{\scriptsize{AC}}}&(\beta,\beta';Y,X) \nu(\df \beta) - 
	\pi_{B|Y,X}(\beta'|Y,X) \bigg| \, \df \beta' \\ & \leq
	\int_{\mathbb{R}} \int_{\mathbb{R}} \big|
	k^{(m)}_{\mbox{\scriptsize{AC}}}(\beta,\beta';Y,X) -
	\pi_{B|Y,X}(\beta'|Y,X) \big| \, \df \beta' \, \nu(\df \beta)
	\\ & \leq \left( \int_{\mathbb{R}} H(\beta) \, \nu(\df \beta)
	\right) \rho^{m-1} \;.
	\end{align*}
	One can verify that $H(\beta)$ can be bounded by polynomial functions. As a result, by Theorem 2.3 in \cite{chen2000propriety}, $\int_{\mathbb{R}} H^2(\beta) \pi_{B|Y,X}(\beta|Y,X) \, \df \beta < \infty$. Then by
	Cauchy-Schwarz, $\int_{\mathbb{R}} H(\beta) \nu(\df\beta)  < \infty$. Therefore, $\rho_{**}(\mathcal{D}_n) \leq
	\rho$ for all sufficiently large~$n\;$.
\end{proof}

As mentioned previously, the convergence rate analyses of
\citet{roy2007convergence} and \citet{chakraborty2016convergence},
which establish the geometric ergodicity of the A\&C chain for fixed
$n$ and $p$, are based on the \textit{un-centered} drift function, $V_0$.  We
end this subsection with a result showing that, while this un-centered
drift may be adequate for \textit{non-asymptotic} results, it simply
does not match the dynamics of the A\&C chain well enough to get a
result like Theorem~\ref{thm:stable-n}.  The following result is
proven is Subsection~\ref{app:instability}.

\begin{proposition}
	\label{prop:instability}
	Assume that (A1) and (A2) hold, and that there exists $\beta_* \in
	\mathbb{R}^p$ such that $\beta_* \neq 0$ and $G(\gamma) =
	G_*(\gamma^T \beta_*)$ for all $\gamma \in \mathbb{R}^p$, where
	$G_*:\mathbb{R} \to (0,1)$ is a strictly increasing function such
	that $G_*(0)=1/2$.  Then, almost surely, any drift and minorization
	based on $V_0(\beta) = \norm{\Sigma^{1/2} \beta}^2$ is
	necessarily unstable in~$n$.
\end{proposition}

\begin{remark}
	In the above proposition, if $G_*(\theta) = \Phi(\theta)$ for all
	$\theta \in \mathbb{R}$, then the probit model is correctly
	specified, and the true parameter is $\beta_*$.
\end{remark}

\subsection{Large $p$, small $n$}
\label{sec:lpsn}

In this subsection, we consider the case where $n$ is fixed and $p$
grows.  In contrast with the strategy of the previous subsection, here
we consider a \textit{deterministic} sequence of data sets.  Also,
since $p$ is changing, we need to specify a sequence of prior
parameters $\{(Q_p,v_p)\}_{p=1}^\infty$.  Let $\mathcal{D}_p =
(v_p,Q_p,X_{n \times p},Y)$, $p \geq 1$, denote a sequence of priors and data sets, where
$Y$ is a fixed $n \times 1$ vector of responses, $X_{n \times p}$ is an $n \times
p$ matrix, $v_p$ is a $p \times 1$ vector, and $Q_p$ is a $p \times p$
positive definite matrix.  (Note that positive definite-ness of $Q_p$
is required for posterior propriety.)  So, each time $p$ increases by
1, we are are given a new $n \times 1$ column vector to add to the
current design matrix.  For the rest of this subsection, we omit the~$p$ and $n \times p$ subscripts.  We also assume that the following conditions hold for
all $p$:
\begin{enumerate}
	\item[($B1$)] $X$ has full row rank;
	\item[($B2$)] There exists a finite, positive constant, $c$, such
	that $\lambda_{\max}(X Q^{-1} X^T) < c$ \;.
\end{enumerate}

Assumption (B1) is equivalent to $X \Sigma^{-1} X^T$ being
non-singular.  Assumption (B2) regulates the eigenvalues of the prior
variance, $Q^{-1}$.  More specifically, it requires that the prior
drives~$B$ towards~$v$.  For illustration, if $X = (1 \, 1 \cdots 1)$,
then (B2) holds if for some $\tau > 0$, $Q^{-1} = \mathrm{diag}
(\tau/p,\tau/p,\dots,\tau/p)$, or $Q^{-1} = \mathrm{diag}
(\tau,\tau/2^2,\dots,\tau/p^2)$.  Assumption (B2) is satisfied by the
generalized $g$-priors used by, e.g., \citet{gupta2007variable},
\citet{ai2009bayesian}, and \citet{baragatti2012study}.  It can be
shown \citep[see, e.g.,][]{chakraborty2016convergence} that (B2) is
equivalent to the existence of a constant $c < 1$ such that
\[
\lambda_{\max} \left( X \Sigma^{-1} X^T \right) < c \;.
\]
While (B2) may seem like a strong assumption, we will provide some
evidence later in this subsection suggesting that it may actually be
necessary.  Here is our main result concerning the large $p$, small
$n$ case.

\begin{theorem}
	\label{thm:stable-p}
	If (B1) and (B2) hold, then there exists a constant $\rho < 1$ such
	that $\rho_*(\mathcal{D}_p) \leq \rho$ for all~$p$.
\end{theorem}

\begin{proof}
	The proof is based on Proposition~\ref{prop:driftmin-p}.  Indeed, as
	in the proof of Theorem~\ref{thm:stable-n}, it suffices to show that
	there exists a $c<1$ such that 
	\[\lambda(\mathcal{D}_p) =
	\lambda^2_{\max} \big ( X \Sigma^{-1} X^T \big ) < c\] 
	for all $p$.
	But this follows immediately from (B2).
\end{proof}

An important feature of Theorem~\ref{thm:stable-p} is that it holds
for any sequence of prior means, $\{v_p\}_{p=1}^\infty$.  This is
achieved by adopting a drift function that is centered around a point
that adapts to the prior mean.  Although Gaussian priors with
non-vanishing means are not commonly used in practice, it is
interesting to see that Albert and Chib's algorithm can be robust
under location shifts in the prior, even when the dimension of the
state space is high.

The following result, which is proven in Subsection~\ref{app:b2nec},
shows that (B2) is not an unreasonable assumption.

\begin{proposition}
	\label{prop:b2nec}
	If $n=1$ and $v=0$, then as $X \Sigma^{-1}X^T$ tends to~$1$,
	\[
	1-\rho_{**} = O\big( 1 - X \Sigma^{-1}X^T \big) \,.
	\]
	In particular,~$\rho_{**}$ is not bounded away from~$1$ if (B2) does not
	hold.
\end{proposition}

\vspace*{6mm}

\noindent {\bf \large Acknowledgment}. The second author's work was supported by NSF
Grant DMS-15-11945.

\vspace*{15mm}

\noindent {\LARGE \bf Appendix}
\begin{appendix}

\vspace*{-3mm}

	\section{Useful Results}
\label{app:results}

\subsection{Hermitian matrices}
\label{app:matrices}

For a Hermitian matrix $M$, let $\lambda_{\min}(M)$ and
$\lambda_{\max}(M)$ denote the smallest and largest eigenvalues of
$M$, respectively.  The following result is part of a famous result due
to H. Weyl \citep[see, e.g.,][Section 4.3]{horn2012matrix}.

\begin{lemma}
	\label{lem:weyl}[Weyl's Inequality]
	Let $M_1$ and $M_2$ be Hermitian matrices of the same size, then
	\begin{equation*}
	\lambda_{\max}(M_2) + \lambda_{\min}(M_1 - M_2) \leq
	\lambda_{\max}(M_1) \leq \lambda_{\max}(M_2) + \lambda_{\max}(M_1 -
	M_2) \;,
	\end{equation*}
	and
	\begin{equation*}
	\lambda_{\min}(M_2) + \lambda_{\min}(M_1 - M_2) \leq
	\lambda_{\min}(M_1) \leq \lambda_{\min}(M_2) + \lambda_{\max}(M_1 -
	M_2) \;.
	\end{equation*}
\end{lemma}

Let $M_1$ and $M_2$ be non-negative definite matrices of the same
size.  It is well known that
\begin{equation*}
\lambda_{\max}(M_1M_2M_1) \leq \lambda_{\max}^2(M_1) \,
\lambda_{\max}(M_2) \;,
\end{equation*}
and
\begin{equation}
\label{eq:m2r}
\lambda_{\min}(M_1M_2M_1) \geq \lambda_{\min}^2(M_1) \,
\lambda_{\min}(M_2) \;.
\end{equation}

The following result follows immediately from Lemma~\ref{lem:weyl}.
\begin{lemma}
	\label{lem:matrixlimit}
	Let $\{M_j\}_{j=1}^\infty$ be a sequence of Hermitian matrices of
	the same size such that $M_j \rightarrow M$.  Then $M$ is Hermitian,
	$\lim_{j \to \infty} \lambda_{\max}(M_j) = \lambda_{\max}(M)$, and
	$\lim_{j \to \infty} \lambda_{\min}(M_j) = \lambda_{\min}(M)$.
\end{lemma}

\subsection{Truncated normal distributions}
\label{app:tn}

Fix $\theta \in \mathbb{R}$.  Let $Z_+ \sim \mbox{TN}(\theta,1;1)$ and
$Z_- \sim \mbox{TN}(\theta,1;0)$.  Then we have
\begin{equation}
\label{eq:evtns}
\mathbb{E}Z_+ = \theta +
\frac{\phi(\theta)}{\Phi(\theta)} \hspace*{5mm}
\mbox{and} \hspace*{5mm} \mathbb{E}Z_- = \theta -
\frac{\phi(\theta)}{1 - \Phi(\theta)} \;.
\end{equation}
Also,
\begin{equation}
\label{eq:vstns}
\mathrm{var}Z_+ = 1 - \frac{\theta \phi(\theta)}{\Phi(\theta)} -
\left( \frac{\phi(\theta)}{\Phi(\theta)} \right)^2 \hspace*{5mm}
\mbox{and} \hspace*{5mm} \mathrm{var}Z_- = 1 +
\frac{\theta\phi(\theta)}{1 - \Phi(\theta)} - \left(
\frac{\phi(\theta)}{1 - \Phi(\theta)} \right)^2 \;.
\end{equation}
Note that
\begin{equation*}
\frac{d\mathbb{E}Z_+}{d\theta} = \mathrm{var}Z_+ \hspace*{5mm}
\mbox{and} \hspace*{5mm} \frac{d\mathbb{E}Z_-}{d\theta} =
\mathrm{var}Z_- \;,
\end{equation*}
which shows that both expectations are increasing in $\theta$.  Recall
that
\[
g(\theta) = \frac{\theta\phi(\theta)}{\Phi(\theta)} +
\left(\frac{\phi(\theta)}{\Phi(\theta)}\right)^2 \;.
\]
Thus, we have the shorthand $\mathrm{var}Z_+ = 1-g(\theta)$ and
$\mathrm{var}Z_- = 1-g(-\theta)$.  The following result is Theorem~3
in \citet{horrace2015moments}.

\begin{lemma}[Horrace (2015)]
	\label{lem:horrace}
	The function $g(\theta)$ is decreasing in~$\theta$.  When $\theta <
	0$, $g(\theta) \in (2/\pi,1)$, and when $\theta > 0$, $g(\theta) \in
	(0,2/\pi)$. In particular, $\mathrm{var}Z_+, \mathrm{var}Z_- \in
	(0,1)$.
\end{lemma}

\section{Proofs}
\label{app:proofs}

\subsection{Corollary~\ref{cor:tv_flip} via Proposition~\ref{prop:tv_flip}}
\label{app:tv_flip}

We begin by resetting the stage from Section~\ref{sec:d_m}.  Let
$\X \subset \mathbb{R}^p$, ${\Y} \subset \mathbb{R}^q$, and
let $\tilde{s}: \X \times {\Y} \rightarrow [0,\infty)$ and
$\tilde{h}: {\Y} \times \X \rightarrow [0,\infty)$ be the
conditional pdfs defined in Section~\ref{sec:d_m}.  Let
$\{\Psi_m\}_{m=0}^\infty$ be a Markov chain on $\X$ whose
Mtd is given by
\[
k(x,x') = \int_{\Y} \tilde{s} \left(x'|\gamma \right)
\tilde{h} \big(\gamma|x \big) \, \df \gamma \;,
\]
and let $\{\tilde{\Psi}_m\}_{m=0}^\infty$ be a Markov chain on ${\Y}$ whose Mtd is given by
\[
\tilde{k}(\gamma,\gamma') = \int_{\X} \tilde{h}
\big(\gamma'|x \big) \tilde{s} \left(x|\gamma \right) \, \df x \;.
\]
Let $K(x,\cdot)$ and~$\Pi$ denote the Mtf and invariant
probability measure for $\{\Psi_m\}_{m=0}^\infty$, respectively, and
define $\tilde{K}(\gamma,\cdot)$ and~$\tilde{\Pi}$ analogously.
Assume that both chains are Harris ergodic (i.e., irreducible,
aperiodic, and Harris recurrent).  The following result provides a
total variation distance connection between the two chains.

\begin{proposition}
	\label{prop:tv_flip}
	Suppose that there exists $R: \Y \times \mathbb{Z}_+ \to (0,\infty)$
	such that for $\tilde{\Psi}_0 \sim \tilde{\nu}$ and $m
	\geq 0$, we have
	\[
	\norm{\tilde{\nu} \tilde{K}^m (\cdot) -
		\tilde{\Pi} (\cdot)}_{\mbox{\tiny{TV}}} \leq \mathbb{E}_{\tilde{\nu}}
	R(\tilde{\Psi}_0, m) \;.
	\]
	Then for $\Psi_0 \sim \nu$ and $m \geq 1$, we have
	\[
	\norm{\nu K^m (\cdot) - \Pi (\cdot)}_{\mbox{\tiny{TV}}} \leq
	\mathbb{E}_{\nu} \left( \int_{\Y} R(\gamma, m-1)
	\tilde{h}(\gamma|\Psi_0) \, \df \gamma \right) \;.
	\]
\end{proposition}

\begin{proof}
	Denote the $m$-step Mtds of the two chains
	as $k^{(m)}(x,x')$ and $\tilde{k}^{(m)}(\gamma,\gamma')$,
	respectively.  Moreover, let $\pi(x)$ and $\tilde{\pi}(\gamma)$ be the
	respective stationary densities for the two chains, and note that
	\[
	\pi(x) = \int_{\Y} \tilde{s}(x|\gamma) \tilde{\pi}(\gamma)
	\, \df \gamma \;.
	\]
	The key is to relate the two $m$-step Mtds.  Indeed, for $m \geq 1$,
	we have
	\[
	k^{(m)}(x,x') = \int_{\Y^2} \tilde{s}(x'|\gamma')
	\tilde{k}^{(m-1)}(\gamma,\gamma') \tilde{h}(\gamma|x) \, \df \gamma \,
	\df \gamma' \;.
	\]
	It follows that for any $m \geq 1$,
	\begin{align*}
	\|\nu K^m (\cdot) &- \Pi (\cdot)\|_{\mbox{\tiny{TV}}} \\
	& = \int_{\X}
	\left| \mathbb{E}_{\nu} k^{(m)}(\Psi_0,x') - \pi(x') \right| \, \df x'
	\\ & \leq \mathbb{E}_{\nu} \int_{\X} \left| k^{(m)}(\Psi_0,x') -
	\pi(x') \right| \, \df x' \\ & = \mathbb{E}_{\nu} \int_{\X} \left|
	\int_{\Y} \tilde{s}(x'|\gamma') \left( \int_{\Y}
	\tilde{k}^{(m-1)}(\gamma,\gamma') \tilde{h}(\gamma|\Psi_0) \, \df
	\gamma - \tilde{\pi}(\gamma') \right) \, \df \gamma' \right| \, \df x'
	\\ & \leq \mathbb{E}_{\nu} \int_{\Y} \left| \int_{\Y}
	\tilde{k}^{(m-1)}(\gamma,\gamma') \tilde{h}(\gamma|\Psi_0) \, \df
	\gamma - \tilde{\pi}(\gamma') \right| \, \df \gamma' \;.
	\end{align*}
	Letting $\tilde{\nu}$ be the probability measure associated
	with $\tilde{h}(\gamma|\Psi_0)$ yields
	\[
	\norm{\nu K^m (\cdot) - \Pi (\cdot)}_{\mbox{\tiny{TV}}} \leq
	\mathbb{E}_{\nu} \norm{\tilde{\nu} K^{m-1} -
		\tilde{\Pi}}_{\mbox{\tiny{TV}}} \leq \mathbb{E}_{\nu} \left(
	\int_{\Y} R(\gamma, m-1) \tilde{h}(\gamma|\Psi_0) \, \df \gamma
	\right).
	\]
\end{proof}

\subsection{An inequality related to Proposition~\ref{prop:V_1}}
\label{app:lambda<1}

Let
\[
\lambda = \sup_{t \in (0,1)} \sup_{\alpha \neq 0}
\frac{\|\Sigma^{-1/2}X^TD(\hat{B} + t\alpha)X \alpha
	\|^2}{\|\Sigma^{1/2}\alpha\|^2} \;.
\]

\begin{proposition}
	$\lambda < 1$.
\end{proposition}
\begin{proof}
	If~$Q$ is positive definite, then by Lemmas~\ref{lem:horrace}
	and~\ref{lem:weyl},
	\[
	\lambda \leq \sup_{\beta \in \mathbb{R}^p} \lambda_{\max}^2 (\Sigma^{-1/2}X^TD(\beta)X\Sigma^{-1/2}) \leq \lambda_{\max}^2(I_p - \Sigma^{-1/2}Q\Sigma^{-1/2}) < 1 \;.
	\]
	For the remainder of the proof, assume that $Q = 0$. 
	
	We begin by introducing some notations.  For $i = 1,2,
	\dots,n$, let
	\[X^*_i = X_i 1_{\{0\}}(Y_i) - X_i
	1_{\{1\}}(Y_i).\]
	Define $A_1, A_2, \dots, A_{2^n}$ to be the
	subsets of $\{1,2,\dots,n\}$. For each $A_j$, define
	\[
	T_j = \big\{ \alpha \neq 0: (X^*_i)^T \alpha \geq 0 \text{ for
		all } i \in A_j, (X^*_i)^T\alpha \leq 0 \text{ for all } i
	\in A_j^c \big \} \;.
	\]
	Let $C = \{j: T_j \neq \emptyset \}$.
	It's easy to see that $\bigcup_{j \in C} T_j =  \mathbb{R}^p\setminus \{0\}$.
	
	Note that $X\Sigma^{-1}X^T \leq I_n$. Thus, for any $\alpha \in \mathbb{R}^p$ and $t \in (0,1)$,
	\[
	\|\Sigma^{-1/2}X^TD(\hat{B} + t\alpha)X \alpha \|^2 \leq \|D(\hat{B} + t\alpha)X \alpha \|^2 \;.
	\]
	Suppose that $\alpha \in T_j$, then by Lemma~\ref{lem:horrace}, for all $i \in A_j$ and $t > 0$, the $i$th diagonal of $D(\hat{B} + t\alpha)$ satisfies
	\begin{align*}
	1 - g(X_i^T(\hat{B} + t\alpha)) 1_{\{1\}}(Y_i)  &- g(-X_i^T(\hat{B} + t\alpha)) 1_{\{0\}}(Y_i) \\ &= 1 - g((-X_i^*)^T(\hat{B} + t\alpha)) \in (0,1-c_0],
	\end{align*}
	where
	\[
	c_0 = \min_{1 \leq i \leq n} g(-(X_i^*)^T\hat{B}) \in (0,1) \;.
	\]
	As a result, whenever $\alpha \in T_j$,
	\[
	\|D(\hat{B} + t\alpha)X\alpha\|^2 \leq \|X\alpha\|^2 - c_0(2-c_0) \sum_{i \in A_j} \|X_i^T\alpha\|^2
	\]
	for all $t \in (0,1)$. It follows that
	\[
	\lambda \leq 1 - c_0(2-c_0)\min_{j \in C} \inf_{\alpha \in T_j} \sum_{i \in A_j} \frac{\|X_i^T\alpha\|^2}{\|\Sigma^{1/2}\alpha\|^2} \;.
	\]
	Let $R_j(\alpha) = \sum_{i \in A_j} \|X_i^T\alpha\|^2/\|\Sigma^{1/2}\alpha\|^2$. Note that for any constant $c \neq 0$, $R_j(c\alpha) = R_j(\alpha)$. Hence, to show that $\lambda < 1$, it suffices to verify that for each $j \in C$, $\inf_{\alpha \in T^*_j} R_j(\alpha) >0$,
	where $T^*_j = \{\alpha \in T_j: \|\alpha\| = 1\}$.
	Since $T^*_j$ is compact, and $R_j(\alpha)$ is continuous in $T^*_j$, it's enough to show that for all $\alpha \in T^*_j$, $R_j(\alpha) > 0$. 
	
	\citet{lesaffre1992existence} showed that $\hat{B}$ exists if and only if the following condition holds:
	\begin{enumerate}
		\item[($C0$)] For every $\alpha \in \mathbb{R}^p \setminus \{0\}$,
		there exists an $i \in \{1,2,\dots,n\}$ such that either $X_i^T
		\alpha \, 1_{\{1\}}(Y_i) < 0$ or $X_i^T \alpha 1_{\{0\}}(Y_i) > 0$.
	\end{enumerate}
	Hence, for any $\alpha \in \mathbb{R}^p$, there exists $i \in \{1,2,\dots,n\}$ such that $(X^*_i)^T\alpha > 0$. This implies that for any $j \in C$ and $\alpha \in T_j^*$, there exists $i \in A_j$ such that $\|X_i^T\alpha\| > 0$. As a result, $R_j(\alpha) > 0$ for any $\alpha \in T_j^*$.
\end{proof}

\subsection{Proposition~\ref{prop:lambdasimple}}
\label{app:lambdasimple}

\vspace*{2mm}
\noindent
{\bf Proposition~\ref{prop:lambdasimple}.\;}
\textit{An upper bound on $\lambda^{1/2}$ in Proposition~\ref{prop:driftmin-n} is}
\[
\lambda_{\max} \left( \Sigma^{-1/2} X^TX
\Sigma^{-1/2} \right) - \frac{2}{\pi} \min_{1 \leq j \leq 2^p}
\lambda_{\min} \left( \Sigma^{-1/2} W(S_j) \Sigma^{-1/2} \right)  \;.
\]

\begin{proof}
	For $\theta \in \{0,1\}$ and $\gamma,\beta \in \mathbb{R}^p$, define
	\[
	u(\gamma;\theta,\beta) = \left\{
	\begin{array}{@{}ll@{}}
	1 & \mbox{if $\theta = 1$ and $\gamma^T\beta > 0$} \\
	1 & \mbox{if $\theta = 0$ and $\gamma^T\beta < 0$} \\
	0 & \text{otherwise} \;.
	\end{array} \right.
	\]
	By Lemma~\ref{lem:horrace}, for any $\beta \in \mathbb{R}^p$,
	\[
	1 - 1_{\{1\}}(Y_i)g\left(X_i^T\beta \right) -
	1_{\{0\}}(Y_i)g\left(-X_i^T\beta \right) \leq 1 - \frac{2}{\pi}\big(
	1 - u(X_i; Y_i, \beta) \big) \;.
	\]
	By Lemma~\ref{lem:weyl},
	\begin{align*}
	\lambda^{1/2} &= \sup_{t \in (0,1)} \sup_{\alpha \neq 0} \frac{\|\Sigma^{-1/2}X^TD(\hat{B} + t\alpha)X\alpha\|}{\|\Sigma^{1/2}\alpha \|} \\
	&\leq \sup_{\beta \in \mathbb{R}^p} \lambda_{\max} \Big( \Sigma^{-1/2}
	X^T D(\beta) X \Sigma^{-1/2} \Big)\\ & \leq \lambda_{\max}
	\left( \Sigma^{-1/2} X^TX \Sigma^{-1/2} \right) \\
	& \hspace{7mm} - \frac{2}{\pi}
	\inf_{\beta \in \mathbb{R}^p} \lambda_{\min} \left\{ \Sigma^{-1/2}
	\left( \sum_{i=1}^{n} X_i \big( 1 - u(X_i;Y_i, \beta) \big) X_i^T
	\right) \Sigma^{-1/2} \right\} \;.
	\end{align*}
	Now, since $u(X_i; Y_i, 0) = 0$,
	\[
	1 - u(X_i;Y_i,0) \geq 1 - u(X_i;Y_i,\beta)
	\]
	for any $\beta \in \mathbb{R}^p$.  Therefore,
	\begin{align*}
	\frac{2}{\pi} \inf_{\beta \in \mathbb{R}^p} \lambda_{\min} & \left[
	\Sigma^{-1/2} \left\{ \sum_{i=1}^{n} X_i \big( 1 - u(X_i;Y_i, \beta)
	\big) X_i^T \right\} \Sigma^{-1/2} \right] \\ & = \frac{2}{\pi}
	\inf_{\beta \neq 0} \lambda_{\min} \left[ \Sigma^{-1/2} \left\{
	\sum_{i=1}^{n} X_i \big( 1 - u(X_i;Y_i, \beta) \big) X_i^T \right\}
	\Sigma^{-1/2} \right] \;.
	\end{align*}
	For $j = 1,2,\dots, 2^p$, let $S'_j$ be the closure of
	$S_j$.  Note that if a vector $\beta \in S'_j \setminus \{0\}$ for
	some~$j$, then for any $\theta \in \{0,1\}$ and $\gamma \in S_j \cup
	(-S_j)$,
	\[
	u(\gamma;\theta,\beta) = \left\{
	\begin{array}{@{}ll@{}}
	1_{\{1\}}(\theta) & \mbox{if $\gamma \in S_j$} \\ 1_{\{0\}}(\theta) &
	\mbox{if $\gamma \in -S_j$} \;.
	\end{array} \right.
	\]
	Since $\bigcup_{j=1}^{2^p} S'_j = \mathbb{R}^p$, by Lemma~\ref{lem:weyl},
	\begin{align*}
	\frac{2}{\pi} \inf_{\beta \neq 0} &\lambda_{\min}  \left[
	\Sigma^{-1/2} \left\{ \sum_{i=1}^{n} X_i \big( 1 - u(X_i;Y_i, \beta)
	\big) X_i^T \right\} \Sigma^{-1/2} \right] \\ & = \frac{2}{\pi}
	\min_{1 \leq j \leq 2^p} \inf_{\beta \in S'_j \setminus \{0\}}
	\lambda_{\min} \left[ \Sigma^{-1/2} \left\{ \sum_{i=1}^{n} X_i \big(
	1 - u(X_i;Y_i, \beta) \big) X_i^T \right\} \Sigma^{-1/2} \right]
	\\ & \geq \frac{2}{\pi} \min_{1 \leq j \leq 2^p} \lambda_{\min}
	\bigg\{ \Sigma^{-1/2} \bigg( \sum_{X_i \in S_j} X_i 1_{\{0\}}(Y_i)
	X_i^T \\ 
	& \hspace{40mm} + \sum_{X_i\in-S_j} X_i 1_{\{1\}}(Y_i) X_i^T \bigg)
	\Sigma^{-1/2} \bigg\} \;.
	\end{align*}
	The result follows immediately.
\end{proof}

\subsection{Assumption (A3) allows for an intercept}
\label{app:intercept}

Suppose that $X_1$ is a $p$-dimensional random vector with $p>1$, such
that the first component of $X_1$ is 1, and the remaining $p-1$
components follow a distribution that admits a pdf with respect to
Lebesgue measure that is positive over the open ball $B_c := \{\gamma
\in \mathbb{R}^{p-1}: \|\gamma\|^2 < c \}$, where $c > 0$.  We will
demonstrate that (A3) is satisfied.  Let $\beta \in \mathbb{R}^p
\setminus \{0\}$ be arbitrary.  The set
\[
\left \{\gamma \in \mathbb{R}^p: \gamma^T \beta = 0 \right\} \cap
\left\{ (1,\gamma^T)^T: \gamma \in B_c \right\}
\]
is in a hyperplane, whose dimension is at most $p-2$.  Meanwhile, the
support of $X_1$'s distribution has positive Lebesgue measure on the
$(p-1)$-dimensional hyperplane $\left\{ (1,\gamma^T)^T: \gamma \in
\mathbb{R}^{p-1} \right\}$.  Thus $\mathbb{P}(X_1^T\beta \neq 0) =
1$. On the other hand, it's easy to verify that $\mathbb{P} \big( X_1
\in S_j \cup (-S_j) \big) > 0$. Thus,
\begin{align*}
\mathbb{P} \Big( (1_{S_j}(X_1) +& 1_{S_j}(-X_1) ) X_1^T\beta
\neq 0 \Big) \\
&= \mathbb{P} \Big( (1_{S_j}(X_1) + 1_{S_j}(-X_1)
) \neq 0 \text{ and } X_1^T\beta \neq 0 \Big) > 0 \;,
\end{align*}
so (A3) holds.

\subsection{Proposition~\ref{prop:proper-n}}
\label{app:proper-n}

\vspace*{2mm}
\noindent
{\bf Proposition~\ref{prop:proper-n}.\;} \textit{Under assumptions
	(A1)-(A3), almost surely, the posterior distribution is proper for
	sufficiently large~$n$.}

\begin{proof}
	It suffices to consider the case that $Q=0$.  Fix $n$ and $p$.
	\citet{chen2000propriety} proved that the posterior is proper
	if the following two conditions are satisfied:
	\begin{enumerate}
		\item[($C1$)] $X$ has full column rank;
		\item[($C3$)] The maximum likelihood estimator of
		$\beta$ exists in $\mathbb{R}^p$.
	\end{enumerate}
	(These two conditions are equivalent to (C1) and (C2) in
	Subsection~\ref{ssec:basics}.)  As mentioned in
	Subsection~\ref{app:lambda<1}, (C3) holds if and only if the
	following condition is satisfied:
	\begin{enumerate}
		\item[($C0$)] For every $\beta \in \mathbb{R}^p
		\setminus \{0\}$, there exists an $i \in
		\{1,2,\dots,n\}$ such that either $X_i^T \beta \,
		1_{\{1\}}(Y_i) < 0$ or $X_i^T \beta 1_{\{0\}}(Y_i) >
		0$.
	\end{enumerate}
	We can therefore prove the result by showing that, almost
	surely, (C1) and (C0) will hold for all sufficiently large
	$n$.  As mentioned previously, (A2) ensures that, almost
	surely, (C1) will hold for all sufficiently large $n$.  Now
	recall that $\{S_j\}_{j=1}^{2^p}$ denotes the set of open
	orthants in $\mathbb{R}^p$.  As before, let $S'_j$ denote the
	closure of $S_j$.  It follows from (A1) and (A3) that, for
	every $j \in \{1,2,\dots,2^p\}$, at least one of the following
	two statements must hold:
	\begin{itemize}
		\item $\mathbb{P}\big( X_1 \in S_j, \; Y_1 = 1 \big) >
		0 \;\; \mbox{and} \;\; \mathbb{P}\big( X_1 \in S_j,
		\; Y_1 = 0 \big) > 0$ \,,
		\item $\mathbb{P}\big( X_1 \in -S_j, \; Y_1 = 1 \big) > 0 \;\;
		\mbox{and} \;\; \mathbb{P}\big( X_1 \in -S_j, \; Y_1 = 0 \big) >
		0$ \,.
	\end{itemize}
	Thus, almost surely, when~$n$ is sufficiently large
	(say $n > N$), for each $j \in \{1,2,\dots,2^p\}$ there exists
	$i',i'' \in \{1,2,\dots,n\}$ such that either $X_{i'}, X_{i''}
	\in S_j$, $(Y_{i'}, Y_{i''}) = (1,0)$, or $X_{i'}, X_{i''} \in
	-S_j$, $(Y_{i'}, Y_{i''}) = (1,0)$.  Now suppose that $n > N$,
	and let $\beta \in \mathbb{R}^p \setminus \{0\}$ be arbitrary.
	Since $\cup_{j=1}^{2^p} (S'_j \setminus \{0\}) = \mathbb{R}^p
	\setminus \{0\}$, there exists an integer $j(\beta)$ such that
	$\beta \in S'_{j(\beta)} \setminus \{0\}$. One can pick~$i'$
	and~$i''$ such that either $X_{i'}, X_{i''} \in S_{j(\beta)}$,
	$(Y_{i'}, Y_{i''}) = (1,0)$, or $X_{i'}, X_{i''} \in
	-S_{j(\beta)}$, $(Y_{i'}, Y_{i''}) = (1,0)$. Note that for any
	$\gamma \in S_{j(\beta)}$, $\gamma^T \beta > 0$. Depending on
	whether $X_{i'}, X_{i''} \in S_{j(\beta)}$ or $X_{i'}, X_{i''}
	\in -S_{j(\beta)}$, (C0) holds with either $i=i''$ or $i=i'$.
\end{proof}

\subsection{Proposition~\ref{prop:instability}}
\label{app:instability}

\vspace*{2mm}

\noindent
{\bf Proposition~\ref{prop:instability}.\;} \textit{Assume that (A1)
	and (A2) hold, and that there exists $\beta_* \in \mathbb{R}^p$ such
	that $\beta_* \neq 0$ and $G(\gamma) = G_*(\gamma^T \beta_*)$ for
	all $\gamma \in \mathbb{R}^p$, where $G_*:\mathbb{R} \to (0,1)$ is a
	strictly increasing function such that $G_*(0)=1/2$.  Then, almost
	surely, any drift and minorization based on $V_0(\beta) =
	\norm{\Sigma^{1/2} \beta}^2$ is necessarily unstable in~$n$.}

\begin{proof}
	We prove the result for the flat prior, and leave the extension to the
	reader.  Note that $V_0(\beta) = \beta^T X^TX \beta$, and assume that
	there exist $\lambda := \lambda(\mathcal{D}_n) < 1$ and $L :=
	L(\mathcal{D}_n) < \infty$ such that
	\[
	\int_{\mathbb{R}^p} V_0(\beta') k_{\mbox{\scriptsize{AC}}}(\beta,\beta') \, \df \beta'
	\leq \lambda V_0(\beta) + L \;.
	\]
	Suppose further that there exists $d := d(\mathcal{D}_n) >
	2L/(1-\lambda)$ such that, for all $\beta$ with $V_0(\beta) < d$,
	\[
	k_{\mbox{\scriptsize{AC}}}(\beta,\beta') \geq \varepsilon q(\beta') \;,
	\]
	where $\varepsilon := \varepsilon(\mathcal{D}_n) > 0$, and $q(\beta')
	:= q(\beta'|Y,X)$ is a pdf on $\mathbb{R}^p$.  We will show that
	$\varepsilon \to 0$ as $n \to \infty$ almost surely.
	
	Our first step is to show $L \to \infty$.  Let $\beta = 0$, then
	$V_0(\beta) = 0$, and
	\begin{align*} 
	\int_{\mathbb{R}^p} V_0(\beta') k_{\mbox{\scriptsize{AC}}}(\beta,&\beta') \, \df \beta'\\
	= p + & \mathrm{tr} \big \{ (X^TX)^{-1}X^T \mathrm{var}(Z|B=0,Y,X) X
	(X^TX)^{-1} \big \} \\ & + \mathbb{E}(Z^T|B=0,Y,X)X (X^TX)^{-1}X^T
	\mathbb{E}(Z|B=0,Y,X) \;.
	\end{align*}
	Hence, when $\beta=0$, a lower bound for $\int_{\mathbb{R}^p}
	V_0(\beta') k_{\mbox{\scriptsize{AC}}}(\beta,\beta') \, \df \beta'$ is given by
	\begin{equation} 
	\label{eq:largeL}
	\begin{aligned}
	&V_0(\beta) + p + \\ &n \Bigg( \frac{1}{n} \sum_{i=1}^{n}
	X_i\mathbb{E}(Z_i|B=0,Y,X) \Bigg)^T \left( \frac{X^TX}{n}
	\right)^{-1} \Bigg( \frac{1}{n} \sum_{i=1}^{n}
	X_i\mathbb{E}(Z_i|B=0,Y,X) \Bigg) \;.
	\end{aligned}
	\end{equation}
	Note that $\{X_i\mathbb{E}(Z_i|B=0,Y,X)\}_{i=1}^n$ are iid. By
	\eqref{eq:evtns}, we have
	\begin{align*}
	\mathbb{E} \big\{ X_1\mathbb{E}(Z_1|B=0,Y,X) \big\} &= \mathbb{E}
	\bigg\{ \bigg( \sqrt{\frac{2}{\pi}} 1_{\{1\}}(Y_1) -
	\sqrt{\frac{2}{\pi}} 1_{\{0\}}(Y_1) \bigg) X_1 \bigg\}\\ &=
	\sqrt{\frac{2}{\pi}} \mathbb{E} \left\{ \left( 2G_*(X_1^T\beta_*) - 1
	\right) X_1 \right\} \,.
	\end{align*}
	Now, since $\mathbb{E}X_1X_1^T$ is positive definite and $\beta_* \neq
	0$, $\beta_*^T \mathbb{E}X_1X_1^T \beta_* > 0$, indicating that
	$\mathbb{P}(X_1^T\beta_* \neq 0) > 0$. It's easy to see that $(
	2G_*(X_1^T\beta_*) - 1 ) \beta_*^TX_1 > 0$ whenever $X_1^T\beta_* \neq
	0$, which implies that $ \beta_*^T \mathbb{E} \left\{ \left(
	2G_*(X_1^T\beta_*) - 1 \right) X_1 \right\} > 0$.  As a result,
	$\mathbb{E} \left\{ \left( 2G_*(X_1^T\beta_*) - 1 \right) X_1 \right\}
	\neq 0$.  Then, by the strong law and (A2), almost surely,
	\begin{align*}
	\lim_{n \to \infty} &\Bigg( \frac{1}{n} \sum_{i=1}^{n} X_i 
	\mathbb{E}(Z_i|B=0,Y,X) \Bigg)^T \\
	& \hspace{25mm} \left( \frac{X^TX}{n} \right)^{-1}
	\Bigg( \frac{1}{n} \sum_{i=1}^{n} X_i\mathbb{E}(Z_i|B=0,Y,X) \Bigg)
	\\  &= \frac{2}{\pi} \, \mathbb{E} \{ ( 2G_*(X_1^T\beta_*)
	- 1 ) X_1 \}^T \big( \mathbb{E} X_1 X_1^T \big)^{-1}
	\mathbb{E} \{ ( 2G_*(X_1^T\beta_*) - 1 ) X_1
	\} > 0 \;.
	\end{align*}
	It then follows from~\eqref{eq:largeL} that there exists a positive
	constant~$c$ such that almost surely, $L > cn$ for any large
	enough~$n$.
	
	Note that $ d > 2L$.  By (A2), $X^TX$ is of order~$n$.  Thus, almost
	surely, there exists a positive constant~$b$ such that for all
	sufficiently large~$n$,
	\[
	\big\{ \beta \in \mathbb{R}^p: \|\beta\|^2 \leq b^2 \big\} \subset
	\big\{ \beta \in \mathbb{R}^p: V_0(\beta) < d \big\} \;.
	\]
	Let $\gamma \in \mathbb{R}^p$ be a fixed vector such that $0 <
	\|\gamma\|^2 \leq b^2$. Then both~$\gamma$ and~$-\gamma$ are in the
	small set $\big\{ \beta \in \mathbb{R}^p: V_0(\beta) < d \big\}$ for
	all sufficiently large~$n$.  We now investigate the conditional
	densities $k_{\mbox{\scriptsize{AC}}}(\gamma,\beta')$ and $k_{\mbox{\scriptsize{AC}}}(-\gamma,\beta')$.  Let
	$(B_1,B_2)$ be a random vector such that the conditional distribution
	of $B_2$ given $B_1=\beta$ is precisely the Mtd
	$k_{\mbox{\scriptsize{AC}}}(\beta,\beta')$. Then
	\[
	\mathbb{E}(B_2|B_1=\beta) = (X^TX)^{-1} \sum_{i=1}^{n} X_i
	\mathbb{E}(Z_i|B=\beta,Y,X) \;.
	\]
	Thus,
	\begin{align}
	\label{eq:EB2B1}
	\big \| \mathbb{E}(B_2|B_1 = \gamma) &- \mathbb{E}(B_2|B_1 =
	-\gamma) \big \|^2 \nonumber \\ & = \bigg\| \left( \frac{X^TX}{n}
	\right)^{-1} \bigg\{ \frac{1}{n} \sum_{i=1}^{n} X_i \big(
	\mathbb{E}(Z_i|B = \gamma,Y,X) \nonumber \\& \hspace{40mm} - \mathbb{E}(Z_i|B = -\gamma,Y,X)
	\big) \bigg\} \bigg\|^2 \;.
	\end{align}
	Now using \eqref{eq:evtns}, we have
	\begin{align*}
	& \mathbb{E}(Z_i|B = \gamma,Y,X) - \mathbb{E}(Z_i|B = -\gamma,Y,X)
	\\ & = \bigg( 2X_i^T\gamma +
	\frac{\phi(X_i^T\gamma)}{\Phi(X_i^T\gamma)} -
	\frac{\phi(X_i^T\gamma)}{1 - \Phi(X_i^T\gamma)} \bigg)
	1_{\{1\}}(Y_i) \\& \hspace{15mm} + \bigg( 2X_i^T\gamma -
	\frac{\phi(X_i^T\gamma)}{1-\Phi(X_i^T\gamma)} +
	\frac{\phi(X_i^T\gamma)}{\Phi(X_i^T\gamma)} \bigg) 1_{\{0\}}(Y_i)
	\\ & = 2X_i^T\gamma + \frac{\phi(X_i^T\gamma)}{\Phi(X_i^T\gamma)} -
	\frac{\phi(X_i^T\gamma)}{1 - \Phi(X_i^T\gamma)} \;.
	\end{align*}
	Another appeal to the strong law yields
	\begin{align*}
	&\lim_{n \to \infty} \frac{1}{n} \sum_{i=1}^{n} X_i \big(
	\mathbb{E}(Z_i|B = \gamma,Y,X)  - \mathbb{E}(Z_i|B = -\gamma,Y,X)
	\big) \\ & = \mathbb{E} \bigg\{ \bigg( 2X_1^T\gamma +
	\frac{\phi(X_1^T\gamma)}{\Phi(X_1^T\gamma)} -
	\frac{\phi(X_1^T\gamma)}{1 - \Phi(X_1^T\gamma)} \bigg) X_1 \bigg \}
	\;.
	\end{align*}
	As in Subsection~\ref{app:tn}, let $Z_+ \sim \mbox{TN}(\theta,1;1)$
	and $Z_- \sim \mbox{TN}(\theta,1;0)$.  Note that
	\[
	2\theta + \frac{\phi(\theta)}{\Phi(\theta)} - \frac{\phi(\theta)}{1 -
		\Phi(\theta)} = \mathbb{E}Z_+ + \mathbb{E}Z_- \;.
	\]
	Since both expectations on the right-hand side are (strictly) increasing
	functions of $\theta$, it follows that the left-hand side is positive
	when $\theta>0$ and negative when $\theta<0$.  Thus,
	\begin{equation}
	\label{eq:misc1}
	\left( 2X_1^T\gamma + \frac{\phi(X_1^T\gamma)}{\Phi(X_1^T\gamma)} -
	\frac{\phi(X_1^T\gamma)}{1 - \Phi(X_1^T\gamma)} \right) \gamma^T X_1
	> 0
	\end{equation}
	whenever $X_1^T \gamma \neq 0$.  Since $\mathbb{E} X_1X_1^T$ is
	positive definite and $\gamma \neq 0$, $\mathbb{P}(X_1^T\gamma \neq 0)
	> 0$.  Taking expectation on both sides of~\eqref{eq:misc1} yields
	\[
	\gamma^T \mathbb{E} \bigg\{ \bigg( 2X_1^T\gamma +
	\frac{\phi(X_1^T\gamma)}{\Phi(X_1^T\gamma)} -
	\frac{\phi(X_1^T\gamma)}{1 - \Phi(X_1^T\gamma)} \bigg) X_1 \bigg \} >
	0 \;,
	\]
	and it follows that
	\[
	\mathbb{E} \bigg\{ \bigg( 2X_1^T\gamma +
	\frac{\phi(X_1^T\gamma)}{\Phi(X_1^T\gamma)} -
	\frac{\phi(X_1^T\gamma)}{1 - \Phi(X_1^T\gamma)} \bigg) X_1 \bigg \}
	\ne 0 \;.
	\]
	By \eqref{eq:EB2B1} and (A2), almost surely,
	\begin{align*}
	\lim_{n \to \infty} &\big \| \mathbb{E}(B_2|B_1 = \gamma)  -
	\mathbb{E}(B_2|B_1 = -\gamma) \big \|^2 \nonumber \\ & = \bigg\|
	\big( \mathbb{E} X_1 X_1^T \big)^{-1} \mathbb{E} \bigg\{ \bigg(
	2X_1^T\gamma + \frac{\phi(X_1^T\gamma)}{\Phi(X_1^T\gamma)} -
	\frac{\phi(X_1^T\gamma)}{1 - \Phi(X_1^T\gamma)} \bigg) X_1 \bigg \}
	\bigg\|^2\\& =: c_\gamma > 0 \;.
	\end{align*}
	Next, consider $\mathrm{var}(B_2|B_1 = \gamma)$ and
	$\mathrm{var}(B_2|B_1 = -\gamma)$. Note that for any $\beta \in
	\mathbb{R}^p$,
	\[
	\mathrm{var}(B_2|B_1=\beta) = (X^TX)^{-1} + (X^TX)^{-1} X^T \,
	\mathrm{var}(Z|B=\beta,Y,X) \, X (X^TX)^{-1} \;.
	\]
	It follows from Lemma~\ref{lem:horrace} that
	$\mathrm{var}(Z_i|B=\beta,Y,X) < 1$ for all $i \in \{1,2,\dots,n\}$.
	Therefore,
	\[
	\mathrm{var}(B_2|B_1=\beta) \leq 2(X^TX)^{-1} \;.
	\]
	Then by Chebyshev's inequality, for any $\delta > 0$,
	\begin{equation*}
	\mathbb{P} \Big( \big\| B_2 - \mathbb{E}(B_2|B_1 = \gamma) \big\|^2
	> \delta \,\Big| B_1 = \gamma \Big) \leq \frac{2}{\delta} \,
	\mathrm{tr} \big\{ (X^TX)^{-1} \big\} \;,
	\end{equation*}
	and
	\begin{equation*}
	\mathbb{P} \Big( \big\| B_2 - \mathbb{E}(B_2|B_1 = -\gamma) \big\|^2
	> \delta \,\Big| B_1 = -\gamma \Big) \leq \frac{2}{\delta} \,
	\mathrm{tr} \big\{ (X^TX)^{-1} \big\} \;.
	\end{equation*}
	Note that by (A2), almost surely, $\mathrm{tr}\,\{(X^TX)^{-1}\} =
	O(n^{-1})$.  We have shown in the previous paragraph that with
	probability~$1$, $ \| \mathbb{E}(B_2|B_1 = \gamma) -
	\mathbb{E}(B_2|B_1 = -\gamma) \|^2 $ converges to a positive constant
	$c_{\gamma}$.  Now, fix $\delta \in (0, c_{\gamma} / 4)$, and let $C_1
	= \big\{ \beta \in \mathbb{R}^p: \|\beta -
	\mathbb{E}(B_2|B_1=\gamma)\|^2 > \delta \big\}$, $C_2 = \big\{ \beta
	\in \mathbb{R}^p: \|\beta - \mathbb{E}(B_2|B_1=-\gamma)\|^2 > \delta
	\big\}$.  With probability~$1$, for all sufficiently large~$n$, $C_1
	\cup C_2 = \mathbb{R}^p$. As a result, almost surely,
	\begin{align*}
	\lim_{n \to \infty} \int_{\mathbb{R}^p} \min \big\{ k_{\mbox{\scriptsize{AC}}}(&\gamma,
	\beta), k_{\mbox{\scriptsize{AC}}}(-\gamma, \beta) \big\} \, \df \beta \\
	&\leq \lim_{n \to
		\infty} \bigg( \int_{C_1} k_{\mbox{\scriptsize{AC}}}(\gamma, \beta) \, \df \beta +
	\int_{C_2} k_{\mbox{\scriptsize{AC}}}(-\gamma, \beta) \, \df \beta \bigg) = 0 \;.
	\end{align*}
	Therefore,
	\[
	\varepsilon \leq \int_{\mathbb{R}^p} \inf_{\{\beta: V_0(\beta) < d\}}
	\left\{ k_{\mbox{\scriptsize{AC}}}(\beta,\beta') \right\} \, \df \beta' \to 0
	\]
	with probability~$1$.
\end{proof}

\subsection{Proposition~\ref{prop:b2nec}}
\label{app:b2nec}

\vspace*{2mm}

\noindent
{\bf Proposition~\ref{prop:b2nec}.\;} \textit{If $n=1$ and $v=0$, then
	as $X \Sigma^{-1}X^T$ tends to~$1$,
	\[
	1 - \rho_{**} = O \big( 1 - X \Sigma^{-1}X^T \big) \,.
	\]
	In particular,~$\rho_{**}$ is not bounded away from~$1$ if (B2) does not
	hold.}

\begin{proof}
	Without loss of generality, assume that $Y = Y_1 = 1$.  The A\&C chain
	has the same convergence rate as the flipped chain defined by the
	following Mtd:
	\begin{align*}
	\check{k}_{\mbox{\scriptsize{AC}}}(z,z') &:= \check{k}_{\mbox{\scriptsize{AC}}}(z,z';Y,X) \\
	& = \int_{\mathbb{R}^p}
	\pi_{Z|B,Y,X}(z'|\beta,Y,X) \pi_{B|Z,Y,X}(\beta|z,Y,X) \, \df \beta
	\;.
	\end{align*}
	The flipped chain is reversible with respect to $\pi_{Z|Y,X}(z|Y,X)$. \pcite{roberts1997geometric} Theorem~2 and
	\pcite{liu:wong:kong:1994} Lemma 2.3 imply that
	\[
	\rho_{**} \geq \gamma_0 := \frac{\mathrm{cov}(Z, Z'|Y,X)}{\mathrm{var}
		(Z|Y,X)} \;,
	\]
	where $Z'|Z=z,Y,X \sim \check{k}_{\mbox{\scriptsize{AC}}}(z,\cdot)$, and $Z|Y,X \sim
	\pi_{Z|Y,X}(\cdot|Y,X)$.  Hence, it suffices to show that
	\[
	1 - \gamma_0 = O\left( 1 - X \Sigma^{-1}X^T \right) \;.
	\]
	We begin by giving an explicit form of $\pi_{Z|Y,X}(z|Y,X)$. Note that
	\begin{align*}
	\pi_{Z|Y,X}(z|Y,X) & = \int_{\mathbb{R}^p} \pi_{Z|B,Y,X}(z|\beta,Y,X)
	\pi_{B|Y,X}(\beta|Y,X) \, \df \beta \\ & \propto 1_{(0,\infty)}(z)
	\int_{\mathbb{R}^p} \exp \bigg\{ -\frac{1}{2} (z-X\beta)^2
	-\frac{1}{2}\beta^TQ\beta \bigg\} \, \df \beta \\ & \propto
	1_{(0,\infty)}(z) \exp \bigg\{ -\frac{1}{2} (1-\psi)z^2 \bigg\} \;,
	\end{align*}
	where $\psi = X \Sigma^{-1} X^T$. Thus, $Z|Y,X \sim \mathrm{TN}(0,
	(1-\psi)^{-1}; 1)$.  It follows immediately from
	\eqref{eq:evtns} and \eqref{eq:vstns} that
	\[
	\mathbb{E}(Z|Y,X) = \mathbb{E}(Z',Y,X) = \sqrt{\frac{2}{\pi}}
	(1-\psi)^{-1/2} \;,
	\]
	and
	\[
	\mathrm{var} (Z|Y,X) = \left(1 -
	\frac{2}{\pi}\right) (1 - \psi)^{-1} \;.
	\]
	To calculate $\mathrm{cov}(Z, Z'|Y,X)$, first note that
	\[
	\mathrm{cov}(Z, Z'|Y,X) = \mathbb{E}(ZZ'|Y,X) - \frac{2}{\pi}
	(1-\psi)^{-1} \;.
	\]
	Now $Z|B,Y,X \sim \mathrm{TN}(XB, 1; 1)$ and $XB|Z,Y,X \sim
	N(\psi Z, \psi)$.  Therefore,
	\begin{align*}
	\mathbb{E}(ZZ'|Y,X) & = \psi \mathbb{E}\left( Z^2|Y,X \right) +
	\mathbb{E} \left( Z \frac{\phi(XB)}{\Phi(XB)} \bigg| Y, X \right)
	\\ & \geq \psi \mathbb{E}\left( Z^2|Y,X \right) \\ &= \psi (1 -
	\psi)^{-1} \;.
	\end{align*}
	Hence,
	\[
	\mathrm{cov}(Z, Z'|Y,X) \geq \Big(1 - \frac{2}{\pi} \Big) (1 -
	\psi)^{-1} - 1 \;.
	\]
	Collecting terms, we have
	\[
	\gamma_0 \geq 1 - \frac{(1-\psi)}{(1-2/\pi)} \;,
	\]
	and the result follows.
\end{proof}

\end{appendix}

\bibliographystyle{ims} 
\bibliography{stable}

\end{document}